\documentclass[11pt]{amsart}
\usepackage[utf8]{inputenc}
\usepackage{amsmath, amssymb}
\usepackage{amsthm}
\usepackage[colorlinks, linkcolor=red, citecolor=blue, urlcolor=blue,  hypertexnames=true]{hyperref}
\usepackage{amsrefs}
\usepackage{setspace,nicefrac, yhmath, amscd,eucal}
\usepackage{color}
\usepackage{a4wide}
\usepackage{comment}

\newtheorem{thm}{Theorem}[section]
\newtheorem{cor}[thm]{Corollary}
\newtheorem{lem}[thm]{Lemma}

\newtheorem{definition}[thm]{Definition}

\newtheorem{remark}[thm]{Remark}
\newtheorem{proposition}[thm]{Proposition}

\newtheorem{defn}[thm]{Definition}
\newtheorem{prop}[thm]{Proposition}
\newtheorem{quest}[thm]{Question}
\newtheorem{prob}[thm]{Problem}

\theoremstyle{definition}

\newcommand{\bM}{\mathsf{M}}
\newcommand{\bA}{\mathsf{A}}
\newcommand{\bN}{\mathsf{N}}
\newcommand{\bP}{\mathsf{P}}
\newcommand{\bH}{\mathsf{H}}
\newcommand{\Ind}{\mathcal{I}}

\newcommand{\Fix}{\textup{Fix}}

\newcommand{\bn}{\mathbb{N}}
\newcommand{\bz}{\mathbb{Z}}
\newcommand{\bc}{\mathbb{C}}
\newcommand{\bq}{\mathbb{Q}}

\DeclareMathOperator{\C}{C}

\newenvironment{rlist}
{
	
	\begin{enumerate}}
	{\end{enumerate}}

\begin{document}

\title{Maximal subgroups and von Neumann subalgebras with the Haagerup property}
\date{}
\author{Yongle Jiang}
\address{Institute of Mathematics of the Polish Academy of Sciences,
	ul.~\'Sniadeckich 8, 00--656 Warszawa, Poland}
\curraddr{School of Mathematical Sciences, Dalian University of Technology, Dalian, 116024, China.} 
\email{yonglejiang@dlut.edu.cn}

\author{Adam Skalski}
\address{Institute of Mathematics of the Polish Academy of Sciences,
	ul.~\'Sniadeckich 8, 00--656 Warszawa, Poland}
\email{a.skalski@impan.pl}

\begin{abstract}
	We initiate a study of maximal subgroups and maximal von Neumann subalgebras which have the Haagerup property. We determine maximal Haagerup subgroups inside $\mathbb{Z}^2 \rtimes SL_2(\mathbb{Z})$ and obtain several explicit instances where maximal Haagerup subgroups yield maximal Haagerup subalgebras. Our techniques are on one hand based on group-theoretic considerations, and on the other on certain results on intermediate von Neumann algebras, in particular these allowing us to deduce that all the intermediate algebras for certain inclusions arise from groups or from group actions. Some remarks and examples concerning maximal non-(T) subgroups and subalgebras are also presented, and we answer two questions of Ge regarding maximal von Neumann subalgebras.
\end{abstract}



\maketitle

The study of maximal von Neumann subalgebras with particular properties has a long and rich history, dating back to the origins of the subject. In particular the role of maximal abelian subalgebras was realised early on by Dixmier and others (see for example \cite{dix}), came to prominence with the groundbreaking results of Feldman and Moore (\cite{femo}) and plays a central role in the modern study of structure and rigidity of von Neumann algebras  (\cite{ioICM}). With time it became clear that similarly one can ask about concrete maximal amenable (in other words injective) von Neumann subalgebras. Here the breakthrough work is due to Popa, who showed in \cite{popa_advances} that the so-called generator masa  in a free group von Neumann algebra is also maximal amenable.

The context for Popa's result is the fact that from the early days of the theory of operator algebras countable discrete groups formed a very rich source of examples; in particular the key theorem of \cite{popa_advances} implies a much easier fact that $\bz < \mathbb{F}_n$ is a maximal amenable subgroup. In recent years there has been a renewed interest in asking when, given a maximal amenable subgroup $H<G$, the von Neumann algebra $L(H)$ is a maximal amenable von Neumann subalgebra of $L(G)$. Satisfactory sufficient conditions, leading to several concrete examples of this phenomenon, were obtained by Boutonnet and Carderi (\cite{remi_carderi1}, \cite{remi_carderi2}). It is also worth noting that some maximal amenable subgroups do not lead to maximal amenable subalgebras.

Another approximation property for von Neumann algebras, originating in the work of Haagerup on free groups (\cite{Haa}), is the Haagerup property (\cite{CoJ}, \cite{Cho}, \cite{Jol}). This again has proved to be very fruitful for the study of operator algebras, partly due to its geometric interpretations, partly as it weakens amenability and yet offers some tools to study the algebra in question, and finally because it forms a strong negation of Kazhdan's Property (T) (see \cite{bdv}, \cite{ccjjv}). In the last decade the Haagerup property played also an important role in the study of quantum groups (\cite{Brannan}, \cite{DFSW}); this motivated the extension of the concept beyond finite von Neumann algebras to arbitrary ones (see \cite{cost} and references therein).

In this work we initiate a study of maximal Haagerup von Neumann subalgebras. The difficulties in approaching this problem are two-fold: first of all relatively little seems to be known on maximal Haagerup subgroups, and secondly the only well-known obstruction to the von Neumann algebraic Haagerup property is relative Property (T) (although see \cite{ci}, where Chifan and Ioana proved that the situation in general is more subtle). Thus we begin our study by analysing examples of maximal Haagerup subgroups in concrete groups without the Haagerup property. In particular we characterise all the maximal Haagerup subgroups in $\bz^2 \rtimes SL_2(\bz)$, showing they are of two types:
\begin{itemize}
	\item $\mathbb{Z}^2\rtimes C$, where $C <SL_2(\bz)$ is a maximal amenable subgroup;
	\item $\{(c(g),g): g \in K\}$, where $K <SL_2(\bz)$ is non-amenable and $c:K \to \bz^2$ is a cocycle which cannot be extended to a larger subgroup.
\end{itemize}
We also record concrete examples of maximal non-(T) subgroups in Property (T) groups.

As we pass to the von Neumann algebraic context, we should stress that our operator algebraic techniques are mostly in a spirit opposite to this appearing in the afore-mentioned work of Popa; in his context, namely for the inclusion $L(\bz)\subset L(\mathbb{F}_n)$, there is no hope to describe all the intermediate von Neumann algebras, whereas our methods in most cases require knowing that the intermediate algebras come from groups, from group actions, or from equivalence relations. Such a requirement might seem at first glance somewhat limiting, but in fact it is quite natural: for example given a Cartan inclusion $\bA \subset \bM$ by \cite{femo} we know not only that $\bM$ can be realised as the von Neumann algebra of an equivalence relation, but also that all the intermediate von Neumann algebras are of this form. Further a version of Galois correspondence (see for example \cite{choda}) says that the intermediate algebras for the inclusion $\bM \subset \bM \rtimes \Gamma$ for an outer action of a discrete group $\Gamma$ on a factor $\bM$ must all come from  subgroups of $\Gamma$. Recent years brought many deeper results of this type, notably these in \cite{suzuki}, \cite{Amr}, \cite{packer}, \cite{cs_2} and \cite{chifan_das}. We will use these, together with certain extensions established here (see say Theorem \ref{thm:Amrutamgeneralised} or Lemma \ref{lem: intermediate subalgebras split}) to exhibit concrete examples of maximal Haagerup subalgebras.

Using such strategy, among other things we show that the following group inclusions have the property that $H$ is a maximal Haagerup subgroup of $G$, and similarly the von Neumann algebra $L(H)$ is a maximal Haagerup von Neumann subalgebra of $L(G)$:
\begin{itemize}
\item $\mathbb{Z}^2\rtimes C<\mathbb{Z}^2\rtimes SL_2(\mathbb{Z})$, where $C <SL_2(\bz)$ is a maximal amenable subgroup such that $\mathbb{Z}^2\rtimes C$ is also ICC;
\item $\Lambda \wr_\Gamma K < \Lambda \wr \Gamma$, where $\Lambda$ is an amenable ICC group, and $K < \Gamma$ is a maximal Haagerup subgroup;
\item $\Lambda \wr SL_2(\bz)< (\Lambda^{\oplus  SL_2(\bz)} \times \bz^2)\rtimes SL_2(\bz)$, where $\Lambda$ is an ICC group with the Haagerup property;
\item $SL_2(\mathbb{Z})*SL_2(\mathbb{Z})<\mathbb{Z}^2\rtimes (SL_2(\mathbb{Z})*SL_2(\mathbb{Z}))$, where the action is defined by first factoring onto the first copy of $SL_2(\bz)$, then combing with the standard matrix multiplication;
\item $\bz \wr_\Gamma K<\bz \wr \Gamma$, where $\Gamma = \bz^2 \rtimes SL_2(\bz)$, $K = \bz^2 \rtimes C$ and $C<SL_2(\bz)$ is a maximal amenable subgroup.
\end{itemize}
The above list is not exhaustive; in particular we obtain also some examples which are related to a general crossed product construction. Some of the new results on intermediate von Neumann algebras should be of use also in some other contexts; it is worth noting that stronger versions of some theorems we prove here (notably on profinite actions) were independently obtained in \cite{chifan_das} and applied in the context of the classification of von Neumann algebras. Here we show how to exploit such results to answer certain questions of  Ge from \cite{ge}.

Many questions related to maximal Haagerup subalgebras remain open, and we list what we believe to be the most important ones in the end of our paper.

The detailed plan of the paper is as follows: in Section 1 we recall the definition of the group-theoretic and von Neumann algebraic Haagerup properties, recall their key features to be used in the sequel and prove some elementary facts on existence of maximal objects. In Section 2 we discuss maximal Haagerup subgroups. After analysing general behaviour of this notion in various products and providing first examples, we ask a question about the existence of Haagerup radicals, understood as largest normal Haagerup subgroups, and identify them inside $ SL_3(\bz)$ and $\bz^2 \rtimes SL_2(\bz)$. Then we prove the first of our main results, characterisation of maximal Haagerup subgroups inside $\bz^2 \rtimes SL_2(\bz)$ and discuss in detail the groups which may appear as such. We finish this section by exhibiting concrete maximal Haagerup subgroups inside $SL_3(\bz)$ and  $\bz^3 \rtimes SL_3(\bz)$. In Section 3 we focus on the von Neumann algebraic context, and produce examples of maximal Haagerup subalgebras using respectively the work of Ioana on ergodic equivalence relations inside $SL_2(\bz) \curvearrowright \mathbb{T}^2$, Galois correspondence of Choda, extremely rigid actions, free products, pro-finite actions and finally roughly normal subgroups. Here we also answer in the positive two questions of Ge regarding maximal von Neumann algebras. Section 4 is devoted to Property (T): we exhibit explicit maximal non-(T) subgroups in Property (T) groups, discuss some cases where maximal (T) or non-(T) subalgebras exist and present a concrete example of a maximal non-(T) von Neumann subalgebra inside a II$_1$-factor with Property (T); it is worth mentioning that other explicit examples of the last instance were obtained in parallel in the article \cite{chifan_das_khan}. Lastly in Section 5 we present a short list of open problems.

All the groups will be discrete and countable; von Neumann algebras will be mostly finite (although in Section 1 we will briefly discuss general $\sigma$-finite von Neumann subalgebras). Inclusions of von Neumann algebras will be always unital, and we will sometimes simply write $\bN<\bM$ if $\bN$ is a von Neumann subalgebra of $\bM$; and similarly $H<G$ if $H$ is a subgroup of $G$. 
If $\bM$ is a von Neumann algebra equipped with a faithful normal state $\phi$ then a von Neumann subalgebra $\bN< \bM$ will be called $\phi$-expected if there exists a $\phi$-preserving (normal) conditional expectation from  $\bM$ onto $\bN$. We say that $H$ is a nontrivial subgroup of $G$ if $\{e\}\neq H \neq G$; a group is nontrivial if it has more than one element. If $H,G$ are groups, then $H^{\oplus G}$ will denote the direct sum of copies of $H$ indexed by $G$, so that we have a natural shift action $G\curvearrowright H^{\oplus G}$ and the corresponding wreath product $H \wr G$. Often we will need the case where $K$ is a subgroup of $G$ acting on  $H^{\oplus G}$ by shifts; then we write the corresponding semidirect product as $H \wr_G K$. The term ICC stands for infinite conjugacy classes.

\section{Haagerup property -- general aspects}

In this section we recall the basic definitions and features of the Haagerup property for groups and von Neumann algebras, which will be used in the rest of the paper.

\subsection*{Groups} The notion of the Haagerup property of a (locally compact) group has its roots in the famous article \cite{Haa}.

\begin{defn}
A group $G$ has the Haagerup property if there exists a sequence of positive-definite functions in $c_0(G)$ which converge to $1$ pointwise.
\end{defn}

On the other hand recall one of the equivalent characterisations of Kazhdan's (relative) Property (T).

\begin{defn}
A group $G$ has the Kazhdan Property (T) (relative to a subgroup $H$) if every  sequence of continuous positive-definite functions on $G$ which converges to $1$ pointwise converges to $1$ uniformly (on $H$).	
\end{defn}

For more information on these two properties we refer to the books \cite{ccjjv} and \cite{bdv}. Note that sometimes for brevity we will simply say that $G$ is Haagerup or $G$ is Kazhdan; or that the inclusion $H<G$ is rigid (meaning that $G$ has Property (T) relative to $H$). These properties are often viewed as strong negations of each other: if $G$ is both Haagerup and Kazhdan, then it must be finite.

\begin{prop}
Suppose that $H$ is a subgroup of a group $G$ and that $H$ has the Haagerup property. Then there exists a maximal Haagerup subgroup of $G$ containing $H$.
\end{prop}
\begin{proof}
This follows by a standard Kuratowski-Zorn argument and the fact that a union of an increasing sequence of discrete countable groups with the Haagerup property is Haagerup (\cite[Proposition 6.1.1]{ccjjv}).
\end{proof}

We record here the analogue of this fact for non-Kazhdan groups (noting also that for obvious reasons it cannot hold for Kazhdan groups: it suffices to consider say $(\bz/2\bz)^{\oplus \infty}$).

\begin{prop}\label{maximalNonT}
	Suppose that $H$ is a subgroup of a group $G$ and that $H$ has does not have Kazhdan's Property (T). Then there exists a maximal subgroup of $G$ containing $H$ and not having Property (T).
\end{prop}
\begin{proof}
To apply the Kuratowski-Zorn argument it suffices to note that Kazhdan groups must be finitely generated; so if there was an increasing sequence of non-Kazhdan groups with the union having Property (T), then the sequence would in fact have to stabilise, which gives a contradiction.
\end{proof}

Finally we recall the key permanence result and the key obstacle for the Haagerup property; the first proposition is \cite[Proposition 6.1.5]{ccjjv}, and the second is an obvious consequence of definitions. These will be used further without any comment.

\begin{prop}\label{extensions}
	Suppose that $H$ is a subgroup of a group $G$. If $H$ has the Haagerup property and the algebra $\ell^{\infty}(G/H)$ admits a $G$-invariant state, then $G$ has the Haagerup property. In particular amenable extensions of Haagerup groups are Haagerup, and if $G$ admits a finite index Haagerup subgroup, then $G$ itself is Haagerup.
	\end{prop}

\begin{prop}	\label{relT}
	If a group $G$ has relative Property (T) with respect to an infinite subgroup, then $G$ is not Haagerup. 
\end{prop}

Naturally Proposition \ref{extensions} remains true if one replaces everywhere the Haagerup property by amenability. Furthermore Proposition \ref{relT} implies that neither $\mathbb{Z}^2\rtimes SL_2(\mathbb{Z})$ nor $SL_3(\mathbb{Z})$ are Haagerup (and the latter one is in fact Kazhdan). We record here a relevant lemma due to Burger (\cite{bur}).

\begin{prop}\label{prop:Burger}
	Suppose that $G < SL_2(\mathbb{Z})$ is a non-amenable subgroup. Then the inclusion $\mathbb{Z}^2<\mathbb{Z}^2\rtimes G$ is rigid, so in particular $\mathbb{Z}^2\rtimes G$ is not Haagerup.
\end{prop}

\subsection*{von Neumann algebras}\label{vNas}

The following definition extends the one given in \cite{CoJ} and then studied for example in \cite {Cho} and \cite{Jol} for finite von Neumann algebras. The formulation below comes from \cite{cs2}; it is equivalent to the one proposed by Okayasu and Tomatsu in \cite{ot}, as shown for example in \cite{cost}. For the terminology `KMS-implementation' we refer to \cite{cs2}; if $\bM$ is a finite von Neumann algebra with a fixed trace it is equivalent to the usual $L^2$-implementation of a given unital completely positive  and trace preserving map on the Hilbert GNS-space.

\begin{defn}
Let $\bM$ be a $\sigma$-finite von Neumann algebra. We say that $\bM$ has the Haagerup property if for some faithful normal state $\phi$ on $\bM$ there exists a sequence of unital completely positive $\phi$-preserving maps whose $KMS$-implementations on the Hilbert space $L^2(\bM, \phi)$ are compact and converge strongly to identity. 
\end{defn}

As shown in \cite{cs} and \cite{cs2} in fact the existence of such maps does not depend on the choice of $\phi$. As expected, the terminology is consistent with that discussed earlier for groups. Choda showed in \cite{Cho} that a group $G$ has the Haagerup property if and only if the von Neumann algebra $L(G)$ has the Haagerup property.

Propositions \ref{extensions} and \ref{relT} have their von Neumann algebraic counterparts. The crossed product case of the theorem below is  \cite[Proposition 3.1]{Jol} (it remains valid also for arbitrary, not necessarily finite, von Neumann algebraic crossed products by amenable groups, as shown in \cite[Theorem 6.6]{cs2} or \cite[Theorem 5.13]{ot}); the general statement is \cite[Theorem 5.1]{bf}.

\begin{thm}
Suppose that $\bM$ is a finite von Neumann algebra with a von Neumann subalgebra $\bN$. If $\bN$ has the Haagerup property and the inclusion $\bN< \bM$ is amenable in the sense of \cite{popa_correspondence}, then $\bM$ has the  Haagerup property. In particular if $\bN$ is Haagerup and $G$ is an amenable group, then $\bN\rtimes G$ is Haagerup.
\end{thm}

The following result follows directly from the definition of relatively rigid von Neumann subalgebras (\cite[Section 4]{popa}). It is worth noting that until recently it was the only known technique of showing that a von Neumann algebra is not Haagerup, but in \cite{ci} Chifan and Ioana, using earlier results of de Cornulier, exhibited an example of a non-Haagerup von Neumann algebra with no relatively rigid diffuse subalgebras.

\begin{prop}\label{prop:rigidincl}
Suppose that $\bM$ is a finite von Neumann algebra with a diffuse von Neumann subalgebra $\bN$ such that the inclusion $\bN <\bM$ is rigid. Then $\bM$ does not have the Haagerup property. 
\end{prop}

The following result in the finite case follows from Theorem 2.3 (ii) in \cite{Jol}.

\begin{lem}
Suppose $\bM$ is a von Neumann algebra equipped with a faithful normal state $\phi$. Let $(\bN_n)_{n \in \bn}$ be an increasing  sequence of $\phi$-expected von Neumann subalgebras of $\bM$ with the Haagerup property. Then the von Neumann algebra $\bN:=(\bigcup_{n \in \bn} \bN_n)''$ is a $\phi$-expected Haagerup von Neumann subalgebra of $\bM$. 
\end{lem}

\begin{proof}
Takesaki's theorem on existence of $\phi$-preserving conditional expectations	(\cite{Takesaki2}) implies that the modular automorphism group leaves each $\bN_n$ globally invariant; the same is then true for $\bN$, so using Takesaki's theorem again we deduce that $\bN$ is $\phi$-expected. Denote the respective $\phi$-preserving conditional expectations by $\mathbb{E}_n: \bM\to \bN_n$,  $\mathbb{E}: \bM\to \bN$. Then the sequence $(\mathbb{E}_n)_{n \in \bN}$ converges pointwise strongly to $\mathbb{E}$; moreover we can view the Hilbert spaces $L^2(\bN_n, \phi)$ as subspaces of $L^2(\bN, \phi)$ and in this picture the KMS-implementations  of $\mathbb{E}_n$ converge strongly to identity on $L^2(\bN, \phi)$ (see for example \cite[Section 2]{JungeXu}). Denote the approximating maps on each of the $\bN_n$ by $(\Phi_k^{(n)})_{k=1}^\infty$; the standard argument using finite subsets of $\bN$ and $\epsilon>0$ allows us to construct an approximating net on $\bN$ out of the maps of the form $\Phi_k^{(n)} \circ \mathbb{E}_n$. 
\end{proof}

Then we have the following corollary (once again arguing via the Kuratowski-Zorn Lemma).

\begin{cor}
Let $\bM$ be a von Neumann algebra, let $\phi$ be a normal state on $\bM$ and assume that $\bN$ is a $\phi$-expected von Neumann subalgebra of $\bM$ with the Haagerup property. There exists a maximal $\phi$-expected von Neumann subalgebra of $\bM$ containing $\bN$ which has the Haagerup property.	In particular if $\bM$ is a finite von Neumann algebra with a Haagerup von Neumann subalgebra $\bN$ then there exists a maximal Haagerup von Neumann subalgebra of $\bM$ containing $\bN$.
\end{cor}

The following is Theorem 3.12 of \cite{ot}.

\begin{prop}
Let a von Neumann algebra $\bM$ be represented on a Hilbert space $\bH$. Then $\bM$ has the Haagerup property if and only if $\bM'$ has the Haagerup property. 	
\end{prop}

The idea of the proof of the next proposition was kindly communicated to us by Yuhei Suzuki; in fact exactly the same proof in the context of injectivity may be found in \cite[Proposition 6.8]{Connes}.
\begin{prop}
Let $\bN < \bM$ be an inclusion of von Neumann algebras, with $\bN$ having the Haagerup property. Assume that $\mathcal{G}$ is a group of unitaries contained in the normaliser $\mathcal{N}_{\bM}(\bN)$. If $\mathcal{G}$ is amenable as a discrete group, then the von Neumann algebra generated by $\bN$ and $\mathcal{G}$ has the Haagerup property. In particular if $u \in \mathcal{N}_{\bM}(\bN)$ then $(\bN	\cup \{u\})''$ has the Haagerup property.
\end{prop}
\begin{proof}
Consider the natural action $\alpha$ of $\mathcal{G}$ on $\bN'$. We then have $(\bN	\cup \mathcal{G})' = \textup{Fix}_{\alpha} (\bN')$. By the previous proposition $\bN'$ is Haagerup. Applying \cite[Corollary 5.13]{ot} we see that $(\bN	\cup  \mathcal{G})'$ is Haagerup and another application of the previous proposition ends the proof.
\end{proof}

\begin{cor}
Maximal Haagerup subalgebras are singular (in other words, if $\bN$ is a maximal Haagerup von Neumann subalgebra of a von Neumann algebra $\bM$ then the normaliser algebra $\mathcal{N}_\bM (\bN)''$ is equal to $\bN$).	
\end{cor}

\section{Maximal Haagerup subgroups}

In this section we discuss various   abstract and concrete results concerning maximal Haagerup subgroups.

\subsection*{Maximal Haagerup subgroups in (Cartesian, free, wreath) products}

We begin by discussing the behaviour  of maximal Haagerup subgroups with respect to certain general constructions.

Consider first the case of the Cartesian product. Given a subgroup $K\subset G_1 \times G_2$ define the first (respectively, second) support subgroup of $K$ as $K_1=\{x \in G_1:\exists_{y \in G_2}\, (x,y) \in K\}$ (respectively, $K_2=\{y \in G_2:\exists_{x \in G_1}\, (x,y) \in K\}$). Then obviously $K\subset K_1 \times K_2$.

\begin{prop}\label{mHapCartesian}
	Let $G_1, G_2$ be groups and suppose that $H_i$ are  maximal Haagerup subgroups in $G_i$ for $i=1, 2$, with at least one of them nontrivial.  Then $H_1\times H_2$ is a maximal Haagerup subgroup inside $G_1\times G_2$. 
\end{prop}
\begin{proof}
Consider a Haagerup subgroup $K\subset G$ containing $H_1 \times H_2$. As $H_1 \subset G_1\cap K$ and the latter group is Haagerup, we have $H_1 = G_1 \cap K$. Similarly $H_2=G_2 \cap K$. Then $H_l$ is normal in $K_l$ for $l=1,2$. Indeed, take  for example $l=1$. For any $x\in K_1$ we have $(x, y)\in K$ for some $y\in K_2$, and hence for each $h\in H_1$, as $(h, e)\in K$, we have $(xhx^{-1}, e)=(x, y)(h, e)(x, y)^{-1}\in K$, so $xhx^{-1}\in H_1$.

Consider then the group $K_1/H_1$ and choose representatives of its elements, say $\{s_i: i \in \Ind\}$, where $\Ind$ is some index set. Fix $i \in \Ind$ and consider $\langle H_1, s_i \rangle$. The latter group is Haagerup, 
 as we have a short exact sequence $1\to H_1\to \langle H_1, s_i \rangle\to \langle H_1, s_i \rangle/H_1\to 1$ and the group $\langle H_1 , s_i\rangle/H_1=H_1\langle s_i \rangle/H_1\cong \langle s_i \rangle/(H_1\cap \langle s_i \rangle)$ is amenable. Then since $H_1$ is maximal Haagerup, $H_1=\langle H_1, s_i \rangle$, so that $s_i\in H_1$. As $i\in \Ind$ is arbitrary, $H_1=K_1$. In a completely identical way we show that $H_2 = K_2$, so that $K=H_1 \times H_2$. 
\end{proof}

It is worth noting that  for amenability a stronger result holds (it is likely known, but we record it here).
\begin{prop}\label{mHap_Cartesian}
	Suppose that $G_1, G_2$ are groups and $K\subset G_1 \times G_2$ is a subgroup. Then the following conditions are equivalent:
	\begin{rlist}
	\item $K$ is a maximal amenable subgroup of $G$;	
	\item $K = H_1 \times H_2$, where $H_i$ is a  maximal amenable subgroup of $G_i$ for $i =1,2$.	
	\end{rlist}
\end{prop}
\begin{proof}
It suffices to observe that if $K$ is an amenable subgroup of $G_1 \times G_2$, then its support subgroups are amenable. Indeed, once we know it, it follows that if $K$ is in fact maximal amenable, it must be equal to $K_1\times K_2$ and the rest is easy.  But the support subgroups are images of $K$ with respect to the homomorphisms given by projections on the first/second coordinate, and as quotients of amenable groups are amenable, the proof is finished.
	\end{proof}

We cannot hope that the analogous result would hold for the Haagerup property, as the next example shows. Before we formulate it, we recall that, as stated in \cite{Th10}, there are two known sources of infinite simple groups with Kazhdan’s Property (T). Such groups appear for example as lattices in certain Kac-Moody groups, see \cite{CR06}. Much earlier, it was also
shown by Gromov (\cite{Gr87}) that every infinite  hyperbolic group surjects onto a Tarski monster group $G$ (that is an infinite group whose every proper subgroup is finite cyclic; in particular a simple group), and $G$ is Kazhdan if only the original hyperbolic group was a Kazhdan group.

\begin{proposition}
	Let $G$ be an infinite simple group with Property (T) and $\phi: \mathbb{F}_n\twoheadrightarrow G$ be a surjective group homomorphism for some $n \in \bn$, $n\geq 2$. Let $\Phi: \mathbb{F}_n\hookrightarrow \mathbb{F}_n\times G$ be defined by $\Phi(s)=(s, \phi(s))$, $s \in \mathbb{F}_n$. Then   $\Phi(\mathbb{F}_n)$ is a maximal Haagerup subgroup of $\mathbb{F}_n\times G$ that does not split as a product of two subgroups of $\mathbb{F}_n$ and $G$.
\end{proposition}
\begin{proof}
	Since $\Phi$ is injective, $\Phi(\mathbb{F}_n)\cong \mathbb{F}_n$ has the Haagerup property. 
	
	Suppose that $g=(x, y)\in \mathbb{F}_n\times G\setminus \Phi(\mathbb{F}_n)$.  Consider $H:=\langle \Phi(\mathbb{F}_n), g \rangle$; we need to show that $H$ cannot be Haagerup.
	Note that $(e, \phi(x)^{-1}y)=(x, \phi(x))^{-1}(x, y)\in H$ and hence $(z, \phi(z))(e, \phi(x)^{-1}y)(z, \phi(z))^{-1}=(e, \phi(z)\phi(x)^{-1}y\phi(z)^{-1})\in H$ for all $z\in \mathbb{F}_n$. 	Since $y\neq \phi(x)$ and $\phi$ is surjective, we deduce that the normal subgroup in $G$ generated by $\phi(x)^{-1}y$, i.e.\ $G$ (as the latter is assumed to be simple), is contained in $H$. Hence $H$ does not have the Haagerup property.
	
	To see that $\Phi(\mathbb{F}_n)$ does not split as a product, just observe that $\Phi(\mathbb{F}_n)\cap \mathbb{F}_n=Ker(\phi)\neq \mathbb{F}_n=\pi_1(\Phi(\mathbb{F}_n))$, where $\pi_1: \mathbb{F}_n\times G\twoheadrightarrow \mathbb{F}_n$ denotes the projection on the first coordinate.
\end{proof}

In Proposition \ref{mHapCartesian} we showed that maximal Haagerup subgroups behave well under taking Cartesian products. The situation is very different for free products and wreath products, as  next two results show.

\begin{prop}
	Let $G_1, G_2$ be nontrivial groups and suppose that $H_i$ are  maximal Haagerup subgroups in $G_i$ for $i=1, 2$, with at least one of them nontrivial.  Then $H_1*H_2$ is not a maximal Haagerup subgroup inside $G_1*G_2$. 
\end{prop}

\begin{proof}
Assume first that both $H_1$ and $H_2$ are nontrivial. 	 Take $s_i\in G_i\setminus H_i$ for $i=1, 2$. Then one can check that $x:=s_1s_2$ is free from $H_1*H_2$, and hence $H_1*H_2*\langle x \rangle$ is a group with the Haagerup property by \cite[Proposition 6.2.3]{ccjjv}; it obviously strictly contains $H_1*H_2$.

In the general case  we  may assume $G_2=H_2$ and $H_1\neq G_1$. Observe that $[G_1:H_1]=\infty$; in particular, we may take $s, t\in G_1\setminus H_1$ such that $H_1s\neq H_1t$. Let $x=sgtg^{-1}$, where $g$ is any nontrivial element from $G_2$. Then we can check that $H_1*G_2$ is free from $x$, hence $\langle x \rangle*H_1*G_2$ is a group with the Haagerup property which strictly contains $H_1*G_2$.
\end{proof}

In \cite[Theorem 1.1]{csv} de Cornulier, Stalder and Valette showed that wreath products $G\wr H$ with $G, H$ Haagerup are Haagerup. 

\begin{lem}
	Let $G_1, G_2$ be groups, with $G_1$ non-trivial, and suppose that $H_i$ are  maximal Haagerup subgroups in $G_i$ for $i=1, 2$. Then $H_1\wr H_2$ is maximal Haagerup in $G_1\wr G_2$ if and only if $G_2=H_2$.
\end{lem}

\begin{proof}
	$\Leftarrow$: Assume that $H_1\wr G_2<K<G_1\wr G_2$ and $K$ has the Haagerup property. Then $K=S\rtimes G_2$ for some $G_2$-invariant subgroup $S$ of $G_1^{\oplus{G_2}}$ such that $H_1^{\oplus{G_2}}\subseteq S$. We claim that $S=H_1^{\oplus{G_2}}$. For any finite subset $F$ of $G_2$, we have $H_1^{\oplus F}\subseteq G_1^{\oplus F}\cap S\subseteq G_1^{\oplus F}$ and $G_1^{\oplus F}\cap S\subset S \subset K$ has the Haagerup property.
	
	By a generalised version of Proposition \ref{mHap_Cartesian} we know that $H_1^{\oplus F}$ is a maximal Haagerup subgroup in $G_1^{\oplus_F}$. Hence $H_1^{\oplus F}=S\cap G_1^{\oplus F}$. 
	Since $F$ is an arbitrary finite set in $G_2$, we deduce that $S=H_1^{\oplus{G_2}}$, i.e.\ $K=H_1\wr G_2$.

	$\Rightarrow$: Suppose that $H_2$ is a nontrivial subgroup of $G_2$ and let $K:=(H_1^{\oplus G_2})\rtimes H_2\cong( \oplus_{z\in H_2\backslash G_2}(\oplus_{g\in z }H_1))\rtimes H_2\cong (\oplus_{z\in H_2\setminus G_2}H_1)\wr H_2\supsetneq (H_1\wr H_2)$,
	where the properness of the last inclusion follows from the fact that $G_1$, hence also $H_1$, is non-trivial. As $K$ has the Haagerup property, we reach a contradiction.
\end{proof}

We record a simple observation, which can be shown using the same ideas as these in the proofs above (and which will be of use to us later).

\begin{prop}\label{prop:strangewreath}
Let $G, K$ be groups, and let $H$ be a subgroup of $G$. If $H$ and $K$ have the Haagerup property, then $K\wr_G H$ has the Haagerup property.
\end{prop}

\begin{proof}
We have 
\[ K\wr_G H = \left(\bigoplus_{g \in G} K\right) \rtimes H = 	\left(\bigoplus_{z \in H\backslash G}  \bigoplus_{g \in z}K \right) \rtimes H  \hookrightarrow \bigoplus_{z \in H\backslash G} \left( (\bigoplus_{g \in z} K) \rtimes H\right) \cong \bigoplus_{z \in H\backslash G} K \wr H,\]
so that the result for usual wreath products from \cite{csv} ends the proof.
\end{proof}

We finish this general subsection with some examples where the wreath/semidirect product constructions yield explicit examples of maximal Haagerup subgroups.

\begin{proposition}\label{prop: wreath product}
	Let $G$ be a countably infinite group with the Haagerup property and let $G_0$ be a quotient of $G$ which does not have the Haagerup property; for example let $G_0$ be an infinite group with property (T) generated by $n$ elements and $G=\mathbb{F}_n$. Let $A$ be a nontrivial abelian group. Then $G$ is a maximal Haagerup subgroup of the generalised wreath product $(A^{\oplus{G_0}})\rtimes G$.
\end{proposition}

\begin{proof}
	First, by \cite[Corollary 3.3]{ci}, we know $(A^{\oplus{G_0}})\rtimes G$ is not Haagerup. We will show that in fact  for any $x\in ((A^{\oplus{G_0}})\rtimes G)\setminus G$  the subgroup $\langle G, x\rangle$ does not have the Haagerup property.
	
	Without loss of generality we may assume that $x\in A^{\oplus{G_0}}$. Consider the subgroup $K:=\langle \tau_g(x): g\in G \rangle\subseteq A^{\oplus{G_0}}$, where $\tau$ is the action used in defining our generalized wreath product. Clearly, $K$ is abelian and $G$-invariant. Then note that we can view $K\rtimes G$ as a subgroup of $(A^{\oplus{G_0}})\rtimes G$ and then $K\rtimes G<\langle G, x\rangle\subseteq K\rtimes G$.  Hence $\langle G, x \rangle=K\rtimes G$. 
	
	Then, since the action $G\curvearrowright K$ factors through the quotient $G_0$, we can  apply \cite[Theorem 3.1(2)]{ci} to deduce that $K\rtimes G$ does not have Haagerup property. Indeed, the stabilizer of $x$ in $G_0$ is contained in $supp(x)supp(x)^{-1}$, hence is finite.
\end{proof}

The last statement can be combined with the action of a free group on the free abelian group, as below.

\begin{proposition} \label{prop:wreathabelian}
Suppose that $m, n \in \bn$ are such that $n\geq 3$ and the free group $G:=\mathbb{F}_m$ admits $G_0:=SL_n(\mathbb{Z})$ as a quotient (implicitly, $m\geq 2$). Let $A$ be a nontrivial abelian group and consider the action $G\curvearrowright \mathbb{Z}^2$ given by some fixed embedding of $\mathbb{F}_m$ into $SL_2(\mathbb{Z})$, e.g. $\mathbb{F}_m\hookrightarrow \mathbb{F}_{\infty}\hookrightarrow \mathbb{F}_2\hookrightarrow SL_2(\mathbb{Z})$ (followed by the matrix multiplication). Then
	$G<(\mathbb{Z}^2\oplus A^{\oplus{G_0}})\rtimes G$ is a maximal Haagerup subgroup.
\end{proposition}

\begin{proof}
	Let $H$ be a subgroup of $ (\mathbb{Z}^2 \oplus A^{\oplus{G_0}})\rtimes G$, strictly containing $G$. We need to show $H$ does not have the Haagerup property.
	Clearly, $H=K\rtimes G$ for some nontrivial subgroup $K<\mathbb{Z}^2\oplus A^{\oplus{G_0}}$ which is $G$-invariant, i.e.\ $G\cdot K=K$.
	
	If $K\cap \mathbb{Z}^2\neq \{e\}$, then we have $G\cdot (K\cap \mathbb{Z}^2)=k\mathbb{Z}^2$ for some $k\in \mathbb{Z} \setminus  \{0\}$ . Then by considering the rigid inclusion  $k\mathbb{Z}^2<k\mathbb{Z}^2\rtimes G$, we see that $k\mathbb{Z}^2\rtimes G$ is not Haagerup, so $H=K\rtimes G$, which contains $k\mathbb{Z}^2\rtimes G$,  is not Haagerup.
	
	If $K\cap A^{\oplus{G_0}}\neq \{e\}$, then we may argue as in the proof of Proposition \ref{prop: wreath product} to deduce that $(K\cap A^{\oplus{G_0}})\rtimes G$ does not have the Haagerup property; so neither does $H$.
	
It remains to consider the case when $K\cap \mathbb{Z}^2=\{e\}$ and $K\cap A^{\oplus{G_0}}=\{e\}$. Note that then $K=\{(x, t(x)): x\in A'\}$ for some bijection $t: A'\cong B'$, where $A'$ is a subgroup of $\mathbb{Z}^2$ and $B'$ is a subgroup of $A^{\oplus{G_0}}$, and both these subgroups are $G$-invariant. The map $t$ is easily seen to be a $G$-equivariant group homomorphism and thus $K\rtimes G\cong A'\rtimes G$ via the isomorphism $\theta: (x, t(x))g\mapsto xg$. 
	
	Since $A'$ is nontrivial and $G$-invariant, $A'\cong k\mathbb{Z}^2$ for some $k\in \mathbb{Z} \setminus  \{0\}$. Then since $\theta$ provides an isomorphism between inclusions $(\theta^{-1}(A')<K\rtimes G)$ and $(A'<A'\rtimes G)$, and the latter inclusion is rigid, we deduce that $H=K\rtimes G$ does not have the Haagerup property.
\end{proof}

\subsection*{Maximal normal Haagerup subgroups}

Although in this article we are mainly interested in maximal Haagerup subgroups (and later maximal Haagerup subalgebras), in the amenable context it is often important to compute the amenable radical of a given group $G$, i.e.\ the largest amenable \emph{normal} subgroup of $G$. This notion was first introduced and studied in \cite{Day}, where Day showed in particular that such a largest subgroup always exists. Recently it has played a big role for example in the study of the unique trace property for group $\C^*$-algebras (see \cite{bkko}). It is then natural to consider the concept of the \emph{Haagerup radical} of a group $G$, i.e.\ the largest normal Haagerup subgroup of $G$. Contrary to the amenable case, it does not seem to be easy to show that every group admits the Haagerup radical; however in the two cases presented below this is the case and moreover the Haagerup radical can be computed.

We begin with a very easy case of $SL_3(\mathbb{Z})$.

\begin{proposition}
	The group $SL_3(\mathbb{Z})$ does not admit any nontrivial normal subgroups with the Haagerup property.
\end{proposition}
\begin{proof}
	Suppose that $H$ is a normal subgroup of $SL_3(\mathbb{Z})$ which has the Haagerup property. By Margulis's normal subgroup theorem (Chapter IV in \cite{mar}), we know that $H$ is either finite or has finite index. Since $H$ is Haagerup, it must be  finite. But then $SL_3(\bz)$ admits no finite normal subgroups (as any finite normal subgroup of a higher rank lattice is contained in its centre, see \cite[Section 17.1]{Morris}).

\end{proof}

\begin{proposition}\label{prop:normalSL2}
	Let $G=\mathbb{Z}^2\rtimes SL_2(\mathbb{Z})$. Then  $G$ admits the Haagerup radical, which is $\mathbb{Z}^2\rtimes \mathbb{Z}/2\mathbb{Z}$, where $\mathbb{Z}/2\mathbb{Z}=\langle -I \rangle$.
\end{proposition}
\begin{proof}
	Let $H$ be a normal subgroup of $G$ with the Haagerup property and $U=\mathbb{Z}\rtimes \mathbb{Z}/2\mathbb{Z}$. We aim to show that $H\subseteq U$. Clearly, $U$ is amenable and normal inside $G$. Thus $HU$ is a normal subgroup of $G$ with the Haagerup property. As $U\subseteq HU$, we see that $HU=\mathbb{Z}^2\rtimes K$ for some group $K$ with $\mathbb{Z}/2\mathbb{Z}\subseteq K\subseteq SL_2(\mathbb{Z})$. Proposition \ref{prop:Burger} implies that $K$ is amenable. As $HU$ is normal in $G$, we know that $K$ is also normal inside $SL_2(\mathbb{Z})$, hence $K$ is contained in the amenable radical of $SL_2(\mathbb{Z})$. 
	
Now observe that the amenable radical of $SL_2(\mathbb{Z})$ coincides with $\mathbb{Z}/2\mathbb{Z}$. This is well-known, but we include a short proof. Recall that $SL_2(\mathbb{Z})$ is hyperbolic and $K$ is contained in a maximal amenable subgroup $C$ of $SL_2(\mathbb{Z})$ which is virtually cyclic and almost malnormal, i.e.\ for all $g\in SL_2(\mathbb{Z})\setminus C$,  $|gCg^{-1}\cap C|<\infty$ (see for example Theorem 7.2 (vi) in \cite{Luck}). Therefore $K$ is finite, and hence elements in $K$ must have eigenvalues equal either $1, 1$ or $-1, -1$, and hence $K=\mathbb{Z}/2\mathbb{Z}$. That is $HU=U$, hence $H\subseteq U$.	
\end{proof}

Note that in the two cases considered above the Haagerup radical coincides with the amenable radical. We end this subsection by observing that the general  question of existence of the Haagerup radical seems to be open even in some apparently elementary cases.

\begin{quest}
Suppose that $H,G$ are groups with the Haagerup property and let $G\times G$ act on $G$ via the left/right shifts. Is it then true that 
$A \wr_G (G \times G)$ has the Haagerup property? Note that $A \wr_G (G \times G)$ is generated by two normal Haagerup subgroups; so its Haagerup radical, if it exists, must be equal to the group itself.
\end{quest}

\subsection*{Maximal Haagerup subgroups in $\mathbb{Z}^2\rtimes SL_2(\mathbb{Z})$}

We are ready to present a central result of this section, i.e.\ a description of all maximal Haagerup subgroups of $\mathbb{Z}^2\rtimes SL_2(\mathbb{Z})$. Recall that if we consider a group $G$ acting on an abelian group $H$ then an $H$-valued 1-cocycle (or just a cocycle) on $G$ is a map $c:G \to H$ such that $c(g_1 g_2) = g_1c(g_2) + c (g_1)$, $g_1,g_2 \in H$, and $c$ is called a coboundary if there is $\xi \in H$ such that $c(g) = \xi - g \xi$, $g \in G$.

\begin{thm}\label{prop: classify mHAP in z2 times sl2z}
	Suppose that $H$ is a maximal Haagerup subgroup in $G=\mathbb{Z}^2\rtimes SL_2(\mathbb{Z})$. Then exactly one of the following cases holds.\\
	(1) $H=\mathbb{Z}^2\rtimes C'$ for some maximal amenable subgroup $C'<SL_2(\mathbb{Z})$.\\
	(2) $H=\{(c(g), g): ~g\in S\}$, where $S<SL_2(\mathbb{Z})$ is a non-amenable subgroup, and $c: S\to \mathbb{Z}^2$ is a cocycle that cannot be extended to a strictly larger subgroup of $SL_2(\mathbb{Z})$. 
	In particular, $H\cong S$, and $H$ is not conjugate to $S$ inside $\mathbb{Z}^2\rtimes SL_2(\mathbb{Z})$ unless $c$ is a coboundary, in which case $S=SL_2(\mathbb{Z})$.
	
	Conversely, each of the subgroups in (1) and (2) is a maximal Haagerup subgroup of $\mathbb{Z}^2\rtimes SL_2(\mathbb{Z})$.
\end{thm}

\begin{proof}
	We split the argument into three cases according to the rank of $H\cap \mathbb{Z}^2$.
	
\textbf{Case 1}: Assume $H\cap \mathbb{Z}^2\cong \mathbb{Z}^2$. We claim that $H=\mathbb{Z}^2\rtimes C'$ for some maximal amenable subgroup $C'< SL_2(\mathbb{Z})$.

Write $C=H\cap SL_2(\mathbb{Z})$ and $H\cap \mathbb{Z}^2=B\mathbb{Z}^2$ for some matrix $B\in M_2(\mathbb{Z})$ with $det(B)\neq 0$. Note that both $H$ and $\mathbb{Z}^2$, hence also $B\mathbb{Z}^2$, are globally invariant under the natural action of $C$, so we can consider the semidirect product $B\mathbb{Z}^2\rtimes C=(H\cap \mathbb{Z}^2)\rtimes C< H$. Observe that $CB\mathbb{Z}^2=B\mathbb{Z}^2$ implies $B^{-1}CB< SL_2(\mathbb{Z})$, hence $B\mathbb{Z}^2\rtimes C\cong \mathbb{Z}^2\rtimes B^{-1}CB$ via the map $(Bx, c)\mapsto (x, B^{-1}cB)$. The fact that $H$ has the Haagerup property implies via Proposition \ref{prop:Burger} that $B^{-1}CB$ is amenable and hence $C$ is an amenable subgroup of $SL_2(\mathbb{Z})$. Since $SL_2(\mathbb{Z})$ is word hyperbolic, $C$ is virtually cyclic (again see for example \cite[Theorem 7.2 (vi)]{Luck}).
	
Now, observe that $K:=\{t\in SL_2(\mathbb{Z}): \exists_{a\in \mathbb{Z}^2}\,  (a,t)\in H\}$ is a subgroup of $SL_2(\mathbb{Z})$ and $H< \mathbb{Z}^2\rtimes K$; therefore, without loss of generality, we may assume $K$ is not amenable; otherwise, $H=\mathbb{Z}^2\rtimes K$, $K$ is automatically maximal amenable and we are done. Now, since $K$ is non-amenable and $SL_2(\mathbb{Z})$ is linear, then Tits' alternative theorem for linear groups implies that $K$ contains $\mathbb{F}_2$, in particular, there exists some $t\in K$ of infinite order. Let then $a\in \mathbb{Z}^2$ be such that $(a, t)\in H$.
	
We claim that there is also some element $t'\in C$ of infinite order.	To see this, for each $n \in \bz$ write $(a,t)^n=(a_n,t^n)$ for some $a_n\in\mathbb{Z}^2$. As we have now the inclusion $H\supseteq (a,t)^{-1}(B\mathbb{Z}^2,e)(a,t)=(t^{-1}\cdot (-a+B\mathbb{Z}^2), t^{-1})(a,t) = (t^{-1}\cdot (B\mathbb{Z}^2), e)$, 
 we conclude that $t^{-1}\cdot(B\mathbb{Z}^2)\subseteq H\cap \mathbb{Z}^2=B\mathbb{Z}^2$. Similarly by considering $H\supseteq (a,t)(B\mathbb{Z}^2,e)(a,t)^{-1}$  we deduce that $t\cdot (B\mathbb{Z}^2)\subseteq B\mathbb{Z}^2$. Thus $t\cdot (B\mathbb{Z}^2)=B\mathbb{Z}^2$. Now, as $[\mathbb{Z}^2: B\mathbb{Z}^2]<\infty$, we can find $n,m \in \bz$, $n\neq m$ such that $a_n\equiv a_m~mod~B\mathbb{Z}^2$, then using $H\ni (a_n,t^n)^{-1}(a_m,t^m)=(t^{-n}\cdot(-a_n+a_m), t^{-n+m})$ and $t^{-n}\cdot(-a_n+a_m)\in t^{-n}(B\mathbb{Z}^2)=B\mathbb{Z}^2\subseteq H$, we deduce that $t^{-n+m}\in C$. Define $t'=t^{-n+m}$.   
	
Now, we claim that $H< \mathbb{Z}^2\rtimes C'$, where $C'$ is a maximal amenable subgroup containing $C$.
	
Suppose this is not the case. Then there exists some $(a', s)\in H$ such that $s\not\in C'$. For any $c\in C$ we have $(a',s)c(a',s)^{-1}=(a'-(scs^{-1})\cdot a', scs^{-1})\in H$, and $[\mathbb{Z}^2: B\mathbb{Z}^2]<\infty$. Considering $c$ of infinite order, whose existence we deduced in the above paragraph, we can find some $n,m \in \bz$, $n\neq m$, such that $a'-(sc^ns^{-1})\cdot a'\equiv a'-(sc^ms^{-1})\cdot a'~~mod~~B\mathbb{Z}^2$. This in turn implies that $sc^{n-m}s^{-1}\in H\cap SL_2(\mathbb{Z})=C$, so that $|sCs^{-1}\cap C|=\infty$ and hence $|sC's^{-1}\cap C'|=\infty$. 

The last statement cannot hold. Indeed by \cite{dgo} we know that any hyperbolic embedded subgroup in a group is almost malnormal. So it suffices to show that $C'$ is hyperbolic embedded in the hyperbolic group $SL_2(\mathbb{Z})$, which is clear by \cite[Corollary 6.6+Theorem 6.8]{dgo}.

	Therefore, $H< \mathbb{Z}^2\rtimes C'$. Since $H$ is a maximal Haagerup subgroup, we deduce that $H=\mathbb{Z}^2\rtimes C'$.

	\textbf{Case 2}: Assume $H\cap \mathbb{Z}^2\cong \mathbb{Z}$. We will deduce a contradiction.
	
Let $x\in H$ be then such that $\langle x\rangle=H\cap \mathbb{Z}^2$. Take any $(a, g)\in H$, where $a\in \mathbb{Z}^2$ and $g\in SL_2(\mathbb{Z})$. Since $x(a,g)x^{-1}(a,g)^{-1}=(x-g\cdot x, e)\in H\cap \mathbb{Z}^2=\langle x \rangle$, we deduce that $g\cdot x\in \langle x \rangle$. Similarly, from $x(a,g)^{-1}x^{-1}(a,g)\in \langle x \rangle$, we deduce that $g^{-1}\cdot x\in \langle x \rangle$. Thus $g\cdot x\in \{\pm x\}$. This in turn means that $H\subseteq \mathbb{Z}^2\rtimes C$, where $C=\{g\in SL_2(\mathbb{Z}): g\cdot x\in \{\pm x\}\}$. Let $C':=\{g\in C: g\cdot x=x\}$. Then $[C: C']\leq 2$ and we observe that $C'$ is cyclic. Indeed, write $x=(m, n)^t$ with $m,n \in \bz $ and $gcd(m, n)=d\neq 0$. There exists some $h\in SL_2(\mathbb{Z})$ such that $h\cdot x=(d, 0)^t$. It follows that $hC'h^{-1}=\{g\in SL_2(\mathbb{Z}): g\cdot (d, 0)^t=(d, 0)^t\}=Id+\mathbb{Z}e_{1, 2}$, which is cyclic. Therefore $C$ is amenable. Finally as $H$ is maximal Haagerup we deduce that $H=\mathbb{Z}^2\rtimes C$, but then $H\cap \mathbb{Z}^2=\mathbb{Z}^2$, which is a desired contradiction.

	\textbf{Case 3}: Assume $H\cap \mathbb{Z}^2$ is trivial. We claim that $H$ is of the form in the second choice in the conclusion.

Clearly, if $(a, g)\in H$, then $a$ is uniquely determined by $g$, so that we can write $a=c(g)$ for some map $c: \pi(H)\to \mathbb{Z}^2$, where $\pi(H)=\{g\in SL_2(\mathbb{Z}): \exists_{b \in \bz^2}\, (b, g)\in H\}$ is a subgroup of $SL_2(\mathbb{Z})$.
Then $H=\{(c(g), g): g\in \pi(H)\}$. As $H\subsetneq \mathbb{Z}^2\rtimes \pi(H)$ is a maximal Haagerup subgroup, $\pi(H)$ is non-amenable. 
Note that $c$ is a cocycle and  $H\cong \pi(H)$ via the map $(c(g), g)\mapsto g$. 
	
Now for any nonamenable subgroup $\pi(H)< SL_2(\mathbb{Z})$, the following are equivalent:
\begin{rlist}
	\item  $H:=\{(c(g), g):~g\in \pi(H)\}$ is maximal Haagerup;
	\item the cocycle $c$ can not be extended to a strictly larger subgroup of $SL_2(\mathbb{Z})$.
\end{rlist}
	Indeed, (i)$\Rightarrow$ (ii) is trivial; to see (ii)$\Rightarrow$(i) holds, assume $H\subsetneq K$ and $K$ is a maximal Haagerup subgroup of $G$. Then since $K\cap\mathbb{Z}^2\not\cong \mathbb{Z}$ by case 2, we deduce $K\cap \mathbb{Z}^2\cong \mathbb{Z}^2$ or $K\cap \mathbb{Z}^2=\{e\}$. If the first choice holds, then $K=\mathbb{Z}^2\rtimes C'$ for some maximal amenable subgroup $C'$ of $SL_2(\mathbb{Z})$ by case 1. But then $\pi(H)<\pi(K)=C'$ is amenable, which is a contradiction. Hence $K\cap\mathbb{Z}^2=\{e\}$ and further $K=\{(c'(g), g): ~g\in \pi(K)\}$ for some cocycle $c': \pi(K)\to\mathbb{Z}^2$. Clearly, $c'|_{\pi(H)}=c$ and $\pi(H)\subsetneq \pi(K)$, so that (ii) cannot hold.

The rest of the theorem follows now by standard arguments (and by Proposition \ref{prop:Burger}).
	\end{proof}

In view of the above theorem it is natural to ask whether one can understand better the subgroups appearing in its conclusion. We begin by analysing maximal amenable subgroups of $SL_2(\bz)$.

\begin{prop}\label{maxamenSL2}
	Let $H<SL_2(\mathbb{Z})$ be an infinite maximal amenable subgroup. Then there exists $g\in SL_2(\mathbb{Z})$ which has infinite order and nonnegative trace such that $\langle g, -I_2\rangle\subseteq H$. Moreover \\
	(1) If $g$ is not similar to $\begin{pmatrix}
	1&1\\
	0&1\\
	\end{pmatrix}$ in $SL_2(\mathbb{R})$ (for example if $Tr(g)>2$), then $\mathbb{Z}^2\rtimes H$ is ICC.;\\
	(2) If $g\sim \begin{pmatrix}
		1&1\\
		0&1\\
	\end{pmatrix}$, then $\mathbb{Z}^2\rtimes H$ is not ICC.
\end{prop}

\begin{proof}
As noted before,  since $SL_2(\mathbb{Z})$ is hyperbolic, $H$ is virtually cyclic. Then note that $\mathcal{Z}(SL_2(\mathbb{Z}))=\{\pm I_2\}\in H$. Then we can take $g$ to be a generator of $\mathbb{Z}< H$. We may assume $Tr(g)\geq 0$; otherwise, we can replace $g$ by $-g$. Now consider the two cases separately.

\vspace*{0.2cm}
Case (1): 

Assume that $g\not\sim \begin{pmatrix}
	1&1\\
	0&1\\
	\end{pmatrix}$. We will show that all points $(x, y)^t\in \mathbb{Z}^2\setminus {(0, 0)^t}$ have infinite orbits under the action of $\langle g\rangle$. Indeed, suppose this is not the case.
	 Then there is $n>0$ such that $g^n(x, y)^t=(x, y)^t$ for some $(0, 0)\neq (x, y)\in\mathbb{Z}^2$, and the two eigenvalues of $g^n$ are equal to 1. Hence either $g^n\sim I_2$ or $g^n\sim \begin{pmatrix}
		1&1\\
		0&1\\
		\end{pmatrix}$. The first case cannot hold, as $g^n=I_2$ would contradict the fact that $g$ has infinite order. Now we just need to show that the second case implies $g\sim\begin{pmatrix}
		1&1\\
		0&1\\
		\end{pmatrix}$.		To see this, 
		note first that we know that the two eigenvalues of $g$ are $\omega^k$ and $\omega^{n-k}=\overline{\omega^k}$,  where $\omega$ is the $n$th primitive root of 1 and $k\in\mathbb{Z}$. We may write $\omega^k=a+ib$. Then since $Tr(g)\in \mathbb{Z}^{\geq 0}$, we know $a=m/2$ for some $ m\in\mathbb{Z}^{\geq 0}$. Further $|a|\leq |\omega^k|=1$ implies $m=0, 1$ or $2$. If $m=0$, the eigenvalues of $g$ are $\pm i$, so that $g\sim \begin{pmatrix}
		i&0\\0&-i\\
		\end{pmatrix}$ so $g^n\not\sim \begin{pmatrix}
		1&1\\
		0&1\\
		\end{pmatrix}
		$.		If $m=1$, the eigenvalues of $g$ are $(1\pm \sqrt{3}i)/2$. Hence $g\sim \begin{pmatrix}
		(1+\sqrt{3}i)/2&0\\
		0&(1-\sqrt{3}i)/2\\
		\end{pmatrix}$, so $g^n\not\sim\begin{pmatrix}
		1&1\\
		0&1\\
		\end{pmatrix}
		$. So $m=2$, and the two eigenvalues of $g$ are equal to 1. Then $g\sim \begin{pmatrix}
		1&1\\
		0&1\\
		\end{pmatrix}$. This is however a contradiction with the assumption in (1). 
We are ready to check that $\mathbb{Z}^2\rtimes H$ is ICC. Let then $(a,h)\in \mathbb{Z}^2\rtimes H$ be any nontrivial element. If $a\neq (0, 0)$, then $\{(g^n a,g^nhg^{-n})=(0,g^n)(a,h)(0,g^{-n}): n\in\mathbb{Z}\}$ is infinite by what we discussed above. Hence we may assume $a=(0, 0)$ and $h\neq e_H=I_2$. To show that the conjugacy class of $(0,h)$ is infinite, it suffices to find a sequence of elements $a_n\in\mathbb{Z}^2$ such that $\{a_n + h(-a_n)\}_{n \in \bn}$ are pairwise distinct.  As $a_n + h(-a_n)=(I_2-h)a_n$ and $I_2-h\neq 0$, we can just take $a_n=(n, 0)^t$ or $(0, n)^t$ depending on which column of $I_2-h$ has nonzero entries.

\vspace*{0.2cm}

Case (2): 

Take any $h\in H$ with infinite order, without loss of generality, we may assume that $Tr(h)\geq 0$. Since $H$ is infinite and virtually cyclic, there exist $k, n\in\mathbb{Z}\setminus\{0\}$ such that $h^n=g^k=P\begin{pmatrix}
	1&k\\
	0&1\\
	\end{pmatrix}P^{-1}$, where $g=P\begin{pmatrix}
	1&1\\
	0&1
	\end{pmatrix}P^{-1}$ for some invertible matrix $P\in M_2(\mathbb{Q})$. 
	
	We claim that $h=P\begin{pmatrix}
	1&k/n\\
	0&1\\
	\end{pmatrix}P^{-1}$.
	To see this, observe that exactly as above we can show that the two eigenvalues of $h$ are of the form $a+ib, a-ib$ with $a=m/2$ for $m=0, 1$ or $2$.	If $m=0$, the eigenvalues of $h$ are $\pm i$, hence $h\sim \begin{pmatrix}
	i&0\\0&-i\\
	\end{pmatrix}$ so $h^n\not\sim \begin{pmatrix}
	1&1\\
	0&1\\
	\end{pmatrix}
	\sim \begin{pmatrix}
	1&k\\
	0&1\\
	\end{pmatrix}$.
	If $m=1$, the eigenvalues of $h$ are $(1\pm \sqrt{3}i)/2$. Hence $h\sim \begin{pmatrix}
	(1+\sqrt{3}i)/2&0\\
	0&(1-\sqrt{3}i)/2\\
	\end{pmatrix}$, so $h^n\not\sim\begin{pmatrix}
	1&1\\
	0&1\\
	\end{pmatrix}
	\sim \begin{pmatrix}
	1&k\\
	0&1\\
	\end{pmatrix}$.
	Thus $m=2$, and the two eigenvalues of $h$ are equal to 1. Then $h\sim \begin{pmatrix}
	1&1\\
	0&1\\
	\end{pmatrix}$ and we may write $h=Q\begin{pmatrix}
	1&1\\
	0&1\\
	\end{pmatrix}Q^{-1}$ for some invertible matrix $Q$. 
	
	Writing $P^{-1}Q=\begin{pmatrix}
	a&b\\
	c&d\\
	\end{pmatrix}$, and plugging in the identity $h^n=P\begin{pmatrix}
	1&k\\0&1\\
	\end{pmatrix}P^{-1}$, we deduce that $c=0$ and $an=kd$. Then another  calculation shows that $h=P\begin{pmatrix}
	1&k/n\\0&1\\
	\end{pmatrix}P^{-1}$.
	
Finally we can show  that $\mathbb{Z}^2\rtimes H$ has a finite conjugacy class. Indeed, there is a non-zero vector $v\in\mathbb{Z}^2$ which is fixed by $g$, just take $v$ to be a suitable multiple of $P(1,0)^t$. Then the above calculation shows that for every $h \in  H$ with infinite order, $hv=\pm v$, which implies that $(v,e)\in \bz^2\rtimes H$ has a finite conjugacy class as $H$ contains only finitely many torsion elements.
\end{proof}

We devote the remaining part of this subsection to understanding better the groups appearing in the case (2) of Theorem \ref{prop: classify mHAP in z2 times sl2z}.
We begin with a simple lemma.

\begin{lem}\label{lem: cocycles from sl2z to z2 are coboundaries}
Every cocycle	$c: SL_2(\mathbb{Z})\to \mathbb{Z}^2$ is a coboundary.
\end{lem}
\begin{proof}
	Fix first a cocycle $c$. It is uniquely determined by the values of $c(s)$ and $c(t)$  where  $s=\begin{pmatrix}
	0&1\\
	-1&0
	\end{pmatrix}$ and $t=\begin{pmatrix}
	0&-1\\
	1&1
	\end{pmatrix}$, so that $\langle s\rangle=\mathbb{Z}/4\mathbb{Z}$, $\langle t \rangle=\mathbb{Z}/6\mathbb{Z}$, and the pair $\{s, t\}$ generates the whole group. 
	We have $c(s^2)=c(t^3)$, i.e.\ $\begin{pmatrix}
	1&1\\
	-1&1
	\end{pmatrix}c(s)=\begin{pmatrix}
	0&-2\\2&2
	\end{pmatrix}c(t)$. So if we set $c(t)=(x, y)\in \mathbb{Z}^2$, we deduce that $c(s)=(-x-2y, x)$. Then a simple calculation shows  that $c$ is a coboundary with $c(g)=\xi-g\xi, g \in SL_2(\bz)$, where $\xi=(-y, x+y)^t$. Note that it suffices to verify the formula on the generators $s,t$.
\end{proof}

We now stop to record a simple group-theoretic corollary of the results of this section; the conclusion itself is likely well-known, but we give a simple proof.

\begin{cor}
	Let $\phi\in Aut(\mathbb{Z}^2\rtimes SL_2(\mathbb{Z}))$. Then there are  $L\in Aut(\mathbb{Z}^2)=GL_2(\mathbb{Z})$ and $\xi\in \mathbb{Z}^2$
	 such that $\phi(a, s)=(L(a)+\xi-(Ls L^{-1}) \xi, LsL^{-1})$ for all $(a,s) \in \mathbb{Z}^2\rtimes SL_2(\mathbb{Z})$. 
\end{cor}

\begin{proof}
	By Proposition \ref{prop:normalSL2}, $\mathbb{Z}^2\rtimes \mathbb{Z}/2\mathbb{Z}$ is a characteristic subgroup, so that we have  $\phi(\mathbb{Z}^2\rtimes \mathbb{Z}/2\mathbb{Z})=\mathbb{Z}^2\rtimes \mathbb{Z}/2\mathbb{Z}$. This implies that $\phi(\mathbb{Z}^2)=\mathbb{Z}^2$ as  $\mathbb{Z}^2$ is the largest nontrivial normal subgroup of $\mathbb{Z}^2\rtimes \mathbb{Z}/2\mathbb{Z}$. Below, we write $L=\phi|_{\mathbb{Z}^2}$. 
	Then, for any $s\in SL_2(\mathbb{Z})$, we write $\phi((0,s))=(c(s), T(s))$ for some maps $c: SL_2(\mathbb{Z})\to \mathbb{Z}^2$ and $T: SL_2(\mathbb{Z})\to SL_2(\mathbb{Z})$. It is routine to check that $T\in Aut(SL_2(\mathbb{Z}))$ and $c\circ T^{-1}$ is a cocycle. As $c\circ T^{-1}$ is always a coboundary by Lemma \ref{lem: cocycles from sl2z to z2 are coboundaries}, we know $c(s)=\xi-T(s) \xi$ for some $\xi\in \mathbb{Z}^2$. Now the homomorphic property of $\phi$ is equivalent to the fact that $T(s)L(b) = L(sb)$ for all $s \in SL_2(\bz)$, $b \in \bz^2$. 
\end{proof}

We will now present some examples of congruence subgroups of $SL_2(\bz)$ (and explicit cocycles) which satisfy the assumptions of statement (2) in Theorem 
\ref{prop: classify mHAP in z2 times sl2z}.

\begin{proposition}
	Let $N \in \bn$, $N \geq 2$, let 
	\[	\Gamma_1(N)=\left\{g\in SL_2(\bz): g_{21} \equiv 0~mod~N, g_{11}\equiv g_{22}\equiv 1~mod~N\right\}\]
and let $\xi=(1/N, 0)^t$. Then the formula $c(g)=\xi-g\xi$ for all $g\in \Gamma_1(N)$ defines a cocycle $c: \Gamma_1(N)\to \mathbb{Z}^2$, which
cannot be extended to any strictly larger subgroup $H<SL_2(\bz)$. Therefore $\{(c(g), g): g\in \Gamma_1(N)\}$ is a maximal Haagerup subgroup in $\mathbb{Z}^2\rtimes SL_2(\mathbb{Z})$.
\end{proposition}
\begin{proof}
	First, recall the definition of some more subgroups in $SL_2(\mathbb{Z})$:
\[\Gamma_0(N)=\left\{g\in SL_2(\mathbb{Z}): g_{21}\equiv 0~mod~N\right\},
		\Gamma(N) = \left\{g\in \Gamma_1(N): g_{12}\equiv 0~mod~N\right\}.
	\]
Clearly, we have $\Gamma(N)\leq \Gamma_1(N)\leq \Gamma_0(N)\leq SL_2(\mathbb{Z})$. Further it is easy to see that the prescription above indeed defines a cocycle on $\Gamma_1(N)$, as \[c\left(\begin{pmatrix}
g_{11}&g_{12}\\g_{21}&g_{22}
\end{pmatrix} \right)= ((1-g_{11})/N, -g_{21}/N)^t.\] 
Suppose that $c$ can  be extended to a larger subgroup $H$ inside $SL_2(\bz)$. Since $\Gamma(N)$ is normal in $SL_2(\mathbb{Z})$, we have $c(sgs^{-1})=\xi-(sgs^{-1})\xi$ for every $s\in SL_2(\mathbb{Z})$ and every $g\in \Gamma(N)$. If now  $s\in H$, we can expand $c(sgs^{-1})$ using cocycle identity to deduce that $(1-sgs^{-1})c(s)=(1-sgs^{-1})(\xi-s\xi)$. Taking any $g\in \Gamma(N)$ such that $1$ is not an eigenvalue of $g$, so that $(1-sgs^{-1})$ is invertible, we conclude that $c(s)=\xi-s\xi$ for every $s\in H$.
	
	Therefore, to get a contradiction, we need to show that for any $s\in SL_2(\mathbb{Z})$, $\xi-s\xi\in \mathbb{Z}^2$ implies $s\in \Gamma_1(N)$. This is a simple calculation based on the  formula displayed above. 

The last statement follows from Theorem 
\ref{prop: classify mHAP in z2 times sl2z}: as $\Gamma_1(N)$ is of finite index in $SL_2(\bz)$, it is clearly non-amenable.
\end{proof}

One can produce other examples as above, using say the free subgroup generated inside $SL_2(\bz)$  by $\begin{pmatrix}
	1&2\\0&1
\end{pmatrix}$ and $\begin{pmatrix}
1&0\\2&1
\end{pmatrix}$ or the subgroup generated by $\begin{pmatrix}
1&2\\0&1
\end{pmatrix}$ and $\begin{pmatrix}
0&1\\-1&0
\end{pmatrix}$. Examples of this type are similar in that they all have finite index in  $SL_2(\bz)$. In fact this turns out to be the case very often, as we have the following proposition.

\begin{prop}
Suppose that $S$ is a non-amenable subgroup of $SL_2(\bz)$ admitting a cocycle $c:S \to \bz^2$ which cannot be extended to a strictly larger subgroup. If either $-I \in S$ or $S$ is finitely generated, then $S$ is of finite index in $SL_2(\bz)$.	
\end{prop}
\begin{proof}
Assume first that $-I \in S$. Suppose that $c:S \to \bz/2\bz$ is a cocycle not admitting a proper extension. Then for each $g\in S$ we have $c(g)+gc(-I)=c(g(-I))=c((-I)g)=c(-I)+(-I)c(g)$, which implies that $c(g)=\frac{(I-g)c(-I)}{2}\in \mathbb{Z}^2$; in particular the value of $c$ at $-I$ determines it uniquely. Write then $c(-I)=(m, n)\in \mathbb{Z}^2$. We will consider then several cases, depending on parity of $m$ and $n$. 
	
	Case 1: Both $m$ and $n$ are even:	then $c(-I)/2\in \mathbb{Z}^2$, so that $c: S\to\mathbb{Z}^2$ is a coboundary and $S=SL_2(\mathbb{Z})$.
	
	Case 2: $m$ is odd and $n$ is even:	in this case, the fact that $c(g)=\frac{(I-g)c(-I)}{2}\in\mathbb{Z}^2$ for all $g\in S$ means that for every $g=\begin{pmatrix}
	g_{11}&g_{12}\\
	g_{21}&g_{22}
	\end{pmatrix}\in S$ the coefficient $g_{11}$ is odd and $g_{21}$ is even. This implies that $\pi(S)\subseteq \pi(H)$, where $\pi: SL_2(\mathbb{Z})\to SL_2(\mathbb{Z}/2\mathbb{Z})$ is the natural homomorphism and $H=\begin{pmatrix}
	1&*\\
	0&1
	\end{pmatrix}\subset SL_2(\mathbb{Z})$. Hence, $S\subseteq Ker(\pi)H$. As the above formula for $c$ defines also a $\bz^2$-valued cocycle on $Ker(\pi)H$, we know that $S=Ker(\pi)H$, hence is of finite index as $Ker(\pi)$ has finite index in $SL_2(\mathbb{Z})$.
	
	Case 3: $m$ is even and $n$ is odd: analogous to Case 2.
	
	Case 4: Both $m$ and $n$ are odd: one can check that in this case, for every $g\in S$, we have $g_{11}+g_{12}\equiv 1 (mod ~2)$ and $g_{21}+g_{22}\equiv 1 (mod ~2)$. Thus, it is clear that the conjugation of $S$ by the matrix $\begin{pmatrix}
	1&0\\1&1
	\end{pmatrix}$ is contained in $Ker(\pi)H$, where $\pi$ and $H$ are defined in Case 2. Therefore, we know that $S$ must be conjugate to the finite index subgroup $Ker(\pi)H$, hence is of finite index itself.

Assume then that $-I \notin S$, but $S$ is finitely generated.	Consider the quotient map $\pi: SL_2(\mathbb{Z})\twoheadrightarrow SL_2(\mathbb{Z})/\{\pm I\}=PSL_2(\mathbb{Z})\cong \mathbb{Z}/2\mathbb{Z}*\mathbb{Z}/3\mathbb{Z}$. Clearly, $S\cong \pi(S)\subseteq \mathbb{Z}/2\mathbb{Z}*\mathbb{Z}/3\mathbb{Z}$. Assume $[SL_2(\bz): S]=\infty$, then $[\mathbb{Z}/2\mathbb{Z}*\mathbb{Z}/3\mathbb{Z}: \pi(S)]=\infty$.  Recall that a group $G$ has the M.~Hall property if every finitely generated subgroup of $G$ is a free factor of some subgroup of $G$ of finite index. In \cite{hall}, M.~Hall proved the non-abelian free groups satisfy this property, and in \cite{burns} R.G.~Burns showed that this property is stable under free products. Thus we know that $\mathbb{Z}/2\mathbb{Z}*\mathbb{Z}/3\mathbb{Z}$ also has the M.~Hall property. Thus, as $\pi(S)$ is finitely generated and of infinite index, there is an element in $PSL_2(\bz)$ which is free from $\pi(S)$. If we now consider any lift of this element to $SL_2(\bz)$, say $g$, we see that it is free from $S$, so that the cocycle $c$ can be extended to $\langle S,g\rangle$. This yields a contradiction.
\end{proof}

We finish this section by discussing an example of an infinite index subgroup $S\subset SL_2(\bz)$ satisfying the assumptions of Theorem \ref{prop: classify mHAP in z2 times sl2z}, case (2).

\begin{prop}
There exists an infinite index non-amenable subgroup $S<SL_2(\bz)$ and a cocycle $c: S \to \bz^2$ such that $c$ does not admit an extension to a larger subgroup.
\end{prop}

\begin{proof}
Denote $a=\begin{pmatrix}
	1&2\\
	0&1
	\end{pmatrix}$ and $b=\begin{pmatrix}
	1&0\\
	2&1
	\end{pmatrix}$. It is known that $\langle a, b\rangle\cong \mathbb{F}_2$ and it has finite index in $SL_2(\mathbb{Z})$.
Consider then  the free group decomposition $\mathbb{F}_2\cong \mathbb{F}_{\infty}\rtimes_{\sigma} \mathbb{Z}$, where  $\mathbb{Z}=\langle b\rangle$, $\mathbb{F}_{\infty}=\langle b^kab^{-k}: k\in \mathbb{Z} \rangle$ and $\sigma_b(b^kab^{-k})=b^{k+1}ab^{-k-1}$, $k \in \bz$. We then view $\mathbb{F}_\infty$ as an infinite index subgroup of $SL_2(\bz)$ and denote it by $S'$.  

Any cocycle $c':S' \to \bz^2$ is uniquely determined by the values $c'(b^k ab^{-k}) =(x_k,y_k) \in \bz^2$, $k \in \bz$; conversely by freeness any choice of $(x_k,y_k)_{k \in \bz}$ determines a cocycle by the above formula. Suppose that such a cocycle can be extended to the subgroup $\langle\mathbb{F}_\infty, b^n\rangle$ inside  $\mathbb{F}_{\infty}\rtimes_{\sigma} \mathbb{Z}$ for some $n \in \bz\setminus\{0\}$. Then $c'(b^nab^{-n})=c'(b^n)+b^nc'(a)-b^nab^{-n}c'(b^n)$; in other words, 
\[(x_n, y_n)^t =   (x_0, 2nx_0 + y_0)^t + \begin{pmatrix} 4n & -2 \\ 8n^2 & -4n \end{pmatrix} c'(b^n)\]
In particular, both $x_n - x_0$ and $y_n - 2nx_0 - y_0$ are even. Thus if we define $x_0=y_0 = 0$, $x_k = y_k = 1$ for $k \in \bz, k \neq 0$, we obtain a cocycle $c':S' \to \bz^2$ which cannot be extended to a larger subgroup of   $\mathbb{F}_{\infty}\rtimes_{\sigma} \mathbb{Z}$. 

Consider then $S'$ again as a subgroup of $SL_2(\bz)$. A standard Kuratowski-Zorn argument (applied to pairs $(d,K)$, where $K$ is a subgroup of $SL_2(\bz)$ containing $S'$ and $d:K \to \bz^2$ is a cocycle extending $c'$) shows that there is a subgroup $S$ and a cocycle $c:S \to \bz^2$ such that $S' \subset S$, $c|_{S'} = c'$, and $c$ does not extend to a strictly larger subgroup. Then $S$ is obviously non-amenable, and moreover it has infinite index. To see the latter, it suffices to note that if it had finite index in $SL_2(\bz)$, then $S \cap \mathbb{F}_2$ would be of finite index in $\mathbb{F}_2$, so in particular would strictly contain $\mathbb{F}_\infty$. That would contradict the fact that the cocycle $c'$ could not be extended inside $\mathbb{F}_2$.
\end{proof}

Naturally $S$ constructed in the above proof cannot contain $-I$ and cannot be finitely generated. We suspect that in fact $S=S'$.

\subsection*{Maximal Haagerup subgroups inside Property (T) groups}

In this subsection we present other explicit examples of maximal Haagerup subgroups, this time inside Property (T) groups, $SL_3(\bz)$ and $\bz^3 \rtimes SL_3(\bz)$. Recall that the latter has Property (T), as noted for example in \cite[Exercise 1.8.7]{bdv}.

\begin{proposition}\label{prop:maxHAPSL3Z}
	Denote by $H(\bz)$ the subgroup of $SL_3(\bz)$ consisting of all upper-triangular matrices. Then $ H(\bz)$ is a maximal Haagerup subgroup of $SL_3(\bz)$.
\end{proposition}
\begin{proof}
	
Denote  by $H(\mathbb{Q})$ the subgroup of $SL_3(\mathbb{Q})$ consisting of all upper triangular matrices (with rational coefficients).

The proof uses the Bruhat decomposition of $SL_3(\mathbb{Q})$ (see \cite[P. 398]{knapp} or \cite{meiri}), i.e.\ the fact that we have $SL_3(\mathbb{Q})=\cup_{\sigma\in Sym(3)}H(\mathbb{Q})p_{\sigma}H(\mathbb{Q})$, with certain $p_\sigma \in SL_3(\bz)$. We begin by listing explicitly all the elements $p_{\sigma}$.

\begin{table}[h!]
	\centering
	\begin{tabular}{|c | c |c | c|c|c|c|}
		\hline
		
		$\sigma$ & id &(12) &(13)&(23)&(123)&(132)\\
		\hline
		$p_{\sigma}$ &id &$\begin{pmatrix}
		0&1&0\\1&0&0\\0&0&-1
		\end{pmatrix}$ &$\begin{pmatrix}
		0&0&-1\\0&1&0\\1&0&0
		\end{pmatrix}$&$\begin{pmatrix}
		1&0&0\\0&0&-1\\0&1&0
		\end{pmatrix}$& $\begin{pmatrix}
		0&0&1\\1&0&0\\0&1&0
		\end{pmatrix}$&$\begin{pmatrix}
		0&1&0\\0&0&1\\1&0&0
		\end{pmatrix}$\\
		\hline
	\end{tabular}
\end{table}

Let $K(\mathbb{Z})_1=\left\{\begin{pmatrix}
A&*\\
0&*
\end{pmatrix}:~ A\in M_2(\mathbb{Z})\right\}\cap SL_3(\mathbb{Z})$ and $K(\mathbb{Z})_2=\left\{\begin{pmatrix}
*&*\\
0&A
\end{pmatrix}:~ A\in M_2(\mathbb{Z})\right\}\cap SL_3(\mathbb{Z})$. Similarly, we define $K(\mathbb{Q})_i$ using $\mathbb{Q}$-coefficients as above for $i=1, 2$. Recall that both $K(\mathbb{Z})_1$ and $K(\mathbb{Z})_2$ have infinite relative (T) subgroups, so do not have the Haagerup property.

Observe the following facts -- the first two are immediate.

\textbf{Fact 1}: $\langle H(\mathbb{Z}), p_{(12)} \rangle=K(\mathbb{Z})_1$.

%
%

\textbf{Fact 2} : $\langle H(\mathbb{Z}), p_{(23)} \rangle=K(\mathbb{Z})_2$.

(The above facts hold also if we replace integer coefficients by rational coefficients.)

\textbf{Fact 3} : For any element $g\in H(\mathbb{Q})$, $p_{(13)}gp_{(13)}\in H(\mathbb{Q})p_{(123)}H(\mathbb{Q})$ if and only if $g_{12}g_{23}\neq 0$ and $g_{13}=0$.

This is a straightforward, but lengthy calculation.  We include a key element of it below.

To see that $\Rightarrow$ holds, write $g=\begin{pmatrix}
x&y&0\\
0&a&b\\
0&0&c
\end{pmatrix}$.  If $by\neq 0$, then 
\[p_{(13)}gp_{(13)}=\begin{pmatrix}
-c&0&0\\
b&a&0\\
0&y&-x
\end{pmatrix}
=\begin{pmatrix}
-by/ac&-y/a&1\\
0&1&0\\
0&0&1
\end{pmatrix}^{-1}p_{(123)}\begin{pmatrix}
b&a&0\\
0&y&-x\\
0&0&-x
\end{pmatrix}.\]
The other implication is easy to see.

\textbf{Fact 4} : For any element $g\in H(\mathbb{Q})$, $p_{(132)}gp_{(132)}\in H(\mathbb{Q})p_{(123)}H(\mathbb{Q})$ if and only if $g_{13}=0$. Similarly $p_{(132)}gp_{(132)}\in H(\mathbb{Q})p_{(13)}H(\mathbb{Q})$ if and only if $g_{13}\neq 0$. 

This is again a lengthy calculation. We have for example, keeping $g$ as above except that we replace $g_{13}$, i.e.\ 0 there by $z$. First, assuming that $z=0$,

\[p_{(132)}gp_{(132)}=\begin{pmatrix}
b&0&a\\
c&0&0\\
z&x&y
\end{pmatrix}=\begin{pmatrix}
1&-b/c&0\\
0&1&0\\
0&0&1
\end{pmatrix}^{-1}p_{(123)}\begin{pmatrix}
c&0&0\\
0&x&y\\
0&0&a
\end{pmatrix}.\]
Then assuming that $z \neq 0$ we have
\[p_{(132)}gp_{(132)}=\begin{pmatrix}
1&-b/c&0\\
0&1&-c/z\\
0&0&1
\end{pmatrix}^{-1}p_{(13)}\begin{pmatrix}
z&x&y\\
0&-cx/z&-cy/z\\
0&0&-a
\end{pmatrix}.\]

Now we are ready to prove that $H(\mathbb{Z})$ is maximal Haagerup inside $SL_3(\mathbb{Z})$.

Let $g\in SL_3(\mathbb{Z})\setminus H(\mathbb{Z})$. Then $g\in H(\mathbb{Q})p_{\sigma}H(\mathbb{Q})\cap SL_3(\mathbb{Z})$ for some $\sigma\neq id$. We will consider all five possibilities.

\textbf{Case 1}: $\sigma=(12)$. 

Then $\langle H(\mathbb{Z}), g \rangle\subseteq \langle H(\mathbb{Q}), g \rangle=\langle H(\mathbb{Q}), p_{(12)} \rangle=K(\mathbb{Q})_1$ (by Fact 1).

Hence, $H(\mathbb{Z})\subsetneq \langle H(\mathbb{Z}), g \rangle\subseteq K(\mathbb{Q})_1\cap SL_3(\mathbb{Z})=K(\mathbb{Z})_1$.

Let $\widetilde{K(\bz)_1}$ denote $\{g \in K(\bz)_1: g_{33}=1\}$. It is a subgroup of $K(\mathbb{Z})_1$ of index 2, isomorphic to $\mathbb{Z}^2\rtimes SL_2(\mathbb{Z})$. Write $\widetilde{H(\bz)}:= H(\bz) \cap \widetilde{K(\bz)_1}$,
$\widetilde{\langle H(\bz),g\rangle}:= \langle H(\bz), g \rangle \cap \widetilde{K(\bz)_1}$. We can use properties of $\mathbb{Z}^2\rtimes SL_2(\mathbb{Z})$ to observe that  $\widetilde{H(\bz)}$ is a maximal Haagerup subgroup of  $\widetilde{K(\bz)_1}$ (as upper triangular matrices are a maximal amenable subgroup of $SL_2(\bz)$). Thus it remains to note that we cannot have $\widetilde{H(\bz)}=\widetilde{\langle H(\bz),g\rangle}$ by simple index considerations.


\textbf{Case 2}: $\sigma=(23)$.

This can be handled similarly as the first case.

\textbf{Case 3}: $\sigma=(123)$.

This case can be handled by an argument appearing in \cite{meiri}. Indeed, consider  an element $g\in H(\mathbb{Q})p_{(123)}H(\mathbb{Q})$.  A direct calculation shows that 
\begin{align*}
g&=\begin{pmatrix}
x&y&z\\
0&a&b\\
0&0&\frac{1}{ax}
\end{pmatrix}p_{(123)}\begin{pmatrix}
x'&y'&z'\\
0&a'&b'\\
0&0&\frac{1}{a'x'}
\end{pmatrix}
=\begin{pmatrix}
yx'&yy'+za'&yz'+zb'+\frac{x}{a'x'}\\
ax'&ay'+ba'&z'a+bb'\\
0&\frac{a'}{ax}&\frac{b'}{ax}
\end{pmatrix} 
\\ &=\begin{pmatrix}
\frac{x}{a'x'}&yx'&yy'+za'\\
0&ax'&ay'+ba'\\
0&0&\frac{a'}{ax}
\end{pmatrix}p_{(123)}\begin{pmatrix}
1&0&\frac{z'}{x'}-\frac{b'y'}{a'x'}\\
0&1&\frac{b'}{a'}\\
0&0&1
\end{pmatrix}   
\end{align*}
where $a,a', b, b', x, x', y, y', z, z' z \in \mathbb{Q}$ are such that $ax'\neq 0$ and $\frac{a'}{ax}\neq 0$. Thus in fact $g \in H(\mathbb{Q})p_{(123)}U_\tau(\mathbb{Q})$, where $U_\tau(\mathbb{Q}) =\{g \in H(\mathbb{Q}):g_{12}=0\}$. Now in \cite[last line in p. 422]{meiri}, it was shown that for such $g$ in fact $[SL_3(\mathbb{Z}): \langle g, e_{1,3}(1)\rangle]<\infty$, where  $e_{1, 3}(1)$ denotes the matrix with 1 on the diagonal and 1 on the (1, 3)-entry. As $e_{1, 3}(1)\in H(\mathbb{Z})$, we see that $[SL_3(\mathbb{Z}): \langle H(\mathbb{Z}), g \rangle]<\infty$.

\textbf{Case 4}: $\sigma=(13)$.

A calculation shows that if $g\in H(\mathbb{Q})p_{(13)}H(\mathbb{Q})$, then $g^{-1}\in H(\mathbb{Q})p_{(13)}H(\mathbb{Q})$. Indeed, this is because $p_{(13)}^{-1}=p_{(13)}\begin{pmatrix}
-1&0&0\\
0&1&0\\
0&0&-1
\end{pmatrix}$.

Write $g=Ap_{(13)}B$, where $A\in H(\mathbb{Q})$ and write $B=\begin{pmatrix}
1&b&c\\
0&1&e\\
0&0&1
\end{pmatrix}\in H(\mathbb{Q})$. Note that here we are using a (finer) Bruhat decomposition as applied in \cite{meiri}. From above, we know that $g^{-1}=B^{-1}p_{(13)}\begin{pmatrix}
-1&0&0\\
0&1&0\\
0&0&-1
\end{pmatrix}A^{-1}\in B^{-1}p_{(13)}H(\mathbb{Q})$.

Now it suffices to show that there exists some $X\in H(\mathbb{Z})$ such that 
\[(BXB^{-1})_{12}\neq 0, (BXB^{-1})_{23}\neq 0~\mbox{and}~(BXB^{-1})_{13}=0,\] 
as then we can argue that $gXg^{-1}\in H(\mathbb{Q})p_{(123)}H(\mathbb{Q})$ by Fact 3 and hence we can deduce that $\langle H(\mathbb{Z}), g \rangle\supseteq \langle H(\mathbb{Z}), gXg^{-1} \rangle$, which does not have the Haagerup property by Case 3.

If we write  $X':=BXB^{-1}$ we aim to find $X\in H(\mathbb{Z})$ such that $X'_{12}X'_{23}\neq 0$ and $X'_{13}=0$.

Calculation shows we can just take $X=\begin{pmatrix}
1&n&m-m'\\
0&1&n'\\
0&0&1
\end{pmatrix}$, where we write $e=\frac{m}{n}, b=\frac{m'}{n'}$ for $m, n, m', n'\in \bz$ with $nn'\neq 0 $.

\textbf{Case 5}: $\sigma=(132)$.

We can use Fact 4 to reduce this case to one of the previous two cases.
Indeed, for any $g\in H(\mathbb{Q})p_{(132)}H(\mathbb{Q})\cap SL_3(\mathbb{Z})$, by Fact 4, we know that $g^2\in H(\mathbb{Q})p_{\sigma'} H(\mathbb{Q})$, where $\sigma'=(13)$ or $(123)$.

Hence $\langle H(\mathbb{Z}), g \rangle\supseteq \langle H(\mathbb{Z}), g^2 \rangle$ does not have the Haagerup property.

\end{proof}

The above proposition yields immediately the fact that $\mathbb{Z}^3\rtimes H(\bz)$ is a maximal Haagerup subgroup of $\mathbb{Z}^3\rtimes SL_3(\bz)$; we state it below and give a second, completely different proof.

\begin{proposition}\label{prop:maxHAPSL3Zprod}
	Denote by $H(\bz)$ the subgroup of $SL_3(\bz)$ consisting of all upper-triangular matrices. Then $\mathbb{Z}^3\rtimes H(\bz)$ is a maximal Haagerup subgroup of $\mathbb{Z}^3\rtimes SL_3(\bz)$.
\end{proposition}

\begin{proof} 
	First note once again that this result is an easy corollary of Proposition \ref{prop:maxHAPSL3Z}; below we present an alternative proof.

	We begin by introducing some more notations: write $\Gamma=SL_3(\bz)$, $G=SL_3(\mathbb{R})$, let $P< G$ be the subgroup of all upper-triangular matrices in $G$, and let $Q:=\{g\in G:g_{21}=g_{31}=0\}$, $Q':=\{g\in G:g_{31}=g_{32}=0\}$.
	Then $\Lambda=\Gamma\cap P$ and $P\subsetneq Q$. Let then $H$ be a subgroup such that $\Lambda\subsetneq H\subseteq \Gamma$; we aim to show that  $\mathbb{Z}^3\rtimes H$ is not Haagerup. We will consider two separate cases.
	
\vspace*{0.2cm}	
	Case (1): there exists a finite index subgroup $H_0 < H$ such that   $H_0<Q\cap \Gamma$ or $H_0< Q'\cap \Gamma$.
	
	It suffices to consider the case $H_0< Q\cap \Gamma$, the other one can be argued analogously. As $Q \cap \Gamma$ has an index 2 subgroup isomorphic to $\mathbb{Z}^2\rtimes SL_2(\mathbb{Z})$, after passing to this subgroup (so changing $H_0$ to another finite index subgroup of $H$), we see that $\mathbb{Z}^2\rtimes \Lambda_2< H_0<\mathbb{Z}^2\rtimes SL_2(\mathbb{Z})$, where $\Lambda_2<SL_2(\mathbb{Z})$ is the maximal amenable subgroup consisting of all upper triangular matrices. If $H_0 \neq \mathbb{Z}^2\rtimes \Lambda_2$, then Proposition \ref{prop:Burger}  implies that $H_0$ (so also $H$) has relative property (T) with respect to an infinite subgroup, and $H$ cannot be Haagerup. It remains to note that if $H$ is strictly larger than $\Lambda$, then $\Lambda$ cannot be of finite index in $H$. This however follows from \cite[Corollary B]{remi_carderi2}, where it is shown that $\Lambda$ is a maximal amenable subgroup of $SL_3(\bz)$.

\vspace*{0.2cm}	
Case (2): no finite index subgroups of $H$ are contained in  $Q\cap \Gamma$ or $Q'\cap \Gamma$.

By \cite[Proposition 7]{bur}, we know that $\mathbb{Z}^3< \mathbb{Z}^3\rtimes H$ has relative Property (T) if and only if there is no $H$-invariant probability measure on the projective space $X:=\mathbb{P}(\mathbb{R}^3):=\mathbb{R}^3-\{0\}/\sim$, where $(a, b, c)^t\sim (a',b',c')^t$ if the two vectors are parallel to each other. Denote by 	 $\pi: SL_3(\mathbb{R})\twoheadrightarrow PSL_3(\mathbb{R})=SL_3(\mathbb{R})/(\mathbb{R}^{\neq 0}\cdot I)$ the natural quotient map and note that $PSL_3(\mathbb{R})$ acts naturally on $X$.

As hinted in \cite[P. 62]{bur}, by \cite[Corollary 3.2.2]{zimmer} it suffices to check that  no finite index subgroup of $\pi(H)$ could fix $[V]$, where $V$ is a subspace of $\mathbb{R}^3$ with dimension 1 or 2 and $[V]$ denotes the image of $V$ in $X$. Suppose then that we have such a finite index subgroup and a subspace $V$. Note that $[V]$ is invariant under $\pi(\Lambda)$. Then either $\textup{dim} V = 1$, so that  $V=(\mathbb{R}, 0, 0)^t$ or  $\textup{dim} V = 2$, so that $V=(\mathbb{R}, \mathbb{R}, 0)^t$. In the first situation the stabilizer subgroup of $[V]$ in $PSL_3(\mathbb{R})$ is equal to $\pi(Q)$. Therefore, $H$ has a finite index subgroup contained in $Q\cap \Gamma$, which contradicts our assumption in Case (2). Similarly in the second situation the stabilizer subgroup of $[V]$ in $PSL_3(\mathbb{R})$ is equal to $\pi(Q')$. Therefore, $H$ has a finite index subgroup contained in $Q'\cap \Gamma$ and we again reach a contradiction.
\end{proof}

\section{Maximal Haagerup von Neumann subalgebras}

In this section we will present several examples of maximal Haagerup von Neumann subalgebras; in most cases (but not all) the proofs will be based on the knowledge of the form of  all intermediate von Neumann subalgebras. 
Some results will be first phrased in a rather general language, but we will always strive to present concrete examples of the form $L(H)< L(G)$, where $H$ is a (neccessarily maximal Haagerup) subgroup of a non-Haagerup group $G$.

\subsection*{Maximal Haagerup subalgebras inside $L(\mathbb{Z}^2\rtimes SL_2(\mathbb{Z}))$}

We begin by the example where we do not know all the intermediate algebras explicitly.

\begin{thm}
	Suppose that $H$ is an infinite maximal amenable subgroup of $SL_2(\mathbb{Z})$ such that $L(\bz^2\rtimes H)$ is a factor; for example let $H$ be a maximal amenable subgroup of $SL_2(\bz)$ containing $\begin{pmatrix} 1& 1 \\1 & 2\end{pmatrix}$. Then $L(\mathbb{Z}^2\rtimes H)<L(\mathbb{Z}^2\rtimes SL_2(\mathbb{Z}))$ is a maximal Haagerup subalgebra.
\end{thm}

\begin{proof}
Consider a von Neumann algebra $\bP$ such that $\bN<\bP<\bM$, where $\bN=L(\mathbb{Z}^2\rtimes H)$, $\bM=L(\mathbb{Z}^2\rtimes SL_2(\mathbb{Z}))$. Begin by noting that since $SL_2(\mathbb{Z})$ is hyperbolic  \cite[Theorem D]{remi_carderi1} or \cite[Theorem A+ Corollary B(1)]{remi_carderi2} show that $\bN$  is a maximal amenable subalgebra of $\bM$. This means in particular that $\bN'\cap \bM \subset \bN$.
Thus, as $\bN$ is a factor, so is $\bP$ (as $\bP'\cap \bP\subseteq \bN'\cap \bM=\bN' \cap \bN = \bc 1$).

It now remains to use the main theorem of \cite{i}, where Ioana showed that when $\bP$ is a subfactor of $\bM$ which contains $L(\bz^2)$, then $\bP$ is either amenable (in which case it equals $\bN$, as discussed above), or the inclusion $L(\bz^2)<\bP$ is rigid, in which case $\bP$ cannot be Haagerup. This ends the proof of the main part of the theorem; the fact that any maximal amenable subgroup of $SL_2(\bz)$ containing $\begin{pmatrix} 1& 1 \\1 & 2\end{pmatrix}$ satisfies the above assumptions follows from Proposition \ref{maxamenSL2}.	
\end{proof}

\subsection*{Examples related to Galois correspondence}

The next example is an almost immediate consequence of the results of \cite{choda}.

\begin{thm}\label{thm:Galois}
	Suppose that $G$ is a group and $\Lambda$ is an amenable ICC group. Then the following conditions are equivalent for a von Neumann algebra $\bP$ such that $L(\Lambda^{\oplus G}) <\bP < L(\Lambda\wr G)$:
	\begin{itemize}
		\item $\bP$ is a maximal Haagerup subalgebra of  $L(\Lambda\wr G)$;
		\item $\bP = L(\Lambda \wr_G H)$ for some maximal Haagerup subgroup $H <G$.
	\end{itemize}
\end{thm}

\begin{proof}
	Consider the Bernoulli action of $G$ on the algebra $\overline{\bigotimes}_G R$, where $R$ is the hyperfinite $II_1$-factor, so that $R\simeq L(\Lambda)$. As the action is strictly outer (see for example \cite[Proposition 4.9]{bb}), we can apply \cite[Corollary 4]{choda} (generalized in \cite[Theorem 5.3]{bb}) to deduce that any intermediate  von Neumann algebra between 
	$\bigotimes_G R$ and $(\overline{\bigotimes}_G R) \rtimes G$ is of the form $\bP_K:=(\overline{\bigotimes}_G R) \rtimes K$, where $K$ is a subgroup of $G$. Now as $L(K) <\bP_K$, if $\bP_K$ is Haagerup, then $K$ must be Haagerup. It remains to note that if $K$ is Haagerup, so is $\Lambda \wr_G K$, which is the content of Proposition \ref{prop:strangewreath}. This ends the proof.
\end{proof}

Concrete instances of the above theorem can be produced for example using Theorem \ref{prop: classify mHAP in z2 times sl2z}.

\subsection*{Extremely rigid actions}

The next class of examples comes from actions which do not admit non-rigid nontrivial quotients.
Recall that an action $G\curvearrowright (X,\mu)$ is called rigid (\cite{popa}) if the inclusion of von Neumann algebras $L^\infty(X, \mu)<L^\infty(X,\mu) \rtimes G$ is rigid; as we recalled in Proposition \ref{prop:rigidincl} unless $(X,\mu)$ is not diffuse, this gives an obstruction to the Haagerup property of $L^\infty(X,\mu) \rtimes G$.

\begin{defn}
	A p.m.p.\ ergodic action  $G\curvearrowright (X,\mu)$ is said to be extremely rigid if the following two conditions are satisfied:
\begin{enumerate}
	\item there are no atomic quotient actions of $G\curvearrowright (X,\mu)$ other than the trivial action;
	\item all the quotient actions of $G\curvearrowright (X,\mu)$ are rigid.
\end{enumerate}
\end{defn}

We first formulate and prove a general result regarding extremely rigid actions. Note that very similar methods are used to show maximal injectivity of certain subalgebras in \cite[Corollary 3.6]{chifan_das}.

\begin{thm}\label{thm: general version for 2nd example}
	Let $\Lambda$ be an ICC group with the Haagerup property and let $G\curvearrowright (X,\mu)$ be a p.m.p.\ ergodic extremely rigid action.
	Moreover, assume that $G$ has the Haagerup property. Let $\bN=\bar{\otimes }_{G} L(\Lambda)$ and consider the Bernoulli action of $G$ on $\bN$. Then $\bN\rtimes G < (\bN\bar{\otimes}L^{\infty}(X,\mu))\rtimes G$ is a maximal Haagerup subalgebra.
\end{thm}

\begin{proof}
First note that $\bN\rtimes G$ has the Haagerup property, as $\bN\rtimes G=L(\Lambda \wr G)$ and $\Lambda \wr G$ has the Haagerup property.

Then, let $\bP$ be any intermediate von Neumann algebra such that $\bN\rtimes G\subsetneq \bP< \bM$, where $\bM= (\bN\bar{\otimes}L^{\infty}(X,\mu))\rtimes G$. Then we claim that $\bP$ does not have the Haagerup property.

Observe that $\bN<\bN\bar{{\otimes}}L^{\infty}(X)$ is centrally $G$-free in the sense of \cite[Theorem 4.3]{suzuki}. Indeed, by \cite[Remark 4.4(3)]{suzuki} (and using the notation of that paper) $\bN^{\omega}\cap (\bN\bar{\otimes}L^{\infty}(X))_{\omega}=\bN^{\omega}\cap (\bN\bar{\otimes}L^{\infty}(X))'=\bN^{\omega}\cap \bN'$ and hence we just need to check $\bN\subset \bN$ is centrally $G$-free, which holds by \cite[Example 4.13]{suzuki}. Therefore, we may apply \cite[Theorem 4.6]{suzuki} to conclude that $\bP=\mathsf{A}\rtimes G$, where $\mathsf{A}$ is a $G$-invariant intermediate subalgebra such that $\bN\subsetneq\mathsf{A}< \bN\bar{\otimes}L^{\infty}(X)$. Since $\bN$ is a finite factor, we deduce that $\mathsf{A}\cong \bN\bar{\otimes }\mathsf{B}$ for some $G$-invariant von Neumann subalgebra $\mathsf{B}$ of $L^{\infty}(X)$ by Ge-Kadison's splitting theorem \cite{gk}.

Then $\bP=(\bN\bar{\otimes}\mathsf{B})\rtimes G$.  As the action is assumed to be extremely rigid, $\bP$ cannot have the Haagerup property as it contains $\mathsf{B}\rtimes G$, $\mathsf{B}$ is diffuse and the inclusion $\mathsf{B} <\mathsf{B} \rtimes G$ is rigid.
\end{proof}

Let us then discuss an example of an extremely rigid action.

\begin{lem}\label{lem: (corrected version) quotients of GL_2(Z)-action}
	Let $G=SL_2(\mathbb{Z})$. The standard (algebraic) action $G\curvearrowright \widehat{\mathbb{Z}^2}\cong \mathbb{T}^2$ is extremely rigid. 
\end{lem}
\begin{proof}
	By \cite[Example 5.9]{witte} (which is essentially a minor correction of \cite[Theorem 2.3]{park}) we know that every nontrivial quotient of the considered action is either $G\curvearrowright \widehat{k\mathbb{Z}^2}$ or $G\curvearrowright \widehat{k\mathbb{Z}^2}/{\sim}$, where $(x,y)\sim (-x, -y)$,  for some $k\geq 1$. Indeed, although in \cite[Example 5.9]{witte}, it is not stated clearly what is the subgroup $\Lambda'$ appearing there,  as a $G$-invariant closed discrete subgroup of $\mathbb{R}^2$ which contains $\mathbb{Z}^2$, $\Lambda'=\mathbb{Z}^2/k$ for some $k\geq 1$. So $G\curvearrowright \Lambda'\setminus\mathbb{R}^2$ just corresponds to the algebraic action $G\curvearrowright \widehat{k\mathbb{Z}^2}$. Then as explained in \cite[Example 5.9]{witte}, we can identify $G\curvearrowright \widehat{k\mathbb{Z}^2}/{\sim}$ with $G\curvearrowright \widehat{k\mathbb{Z}^2}/({\mathbb{Z}/2\mathbb{Z}})$, where $\mathbb{Z}/2\mathbb{Z}\curvearrowright \widehat{k\mathbb{Z}^2}$ as $(x, y)\mapsto (-x, -y)$.
	
	As the inclusion $(k\mathbb{Z}^2<k\mathbb{Z}^2\rtimes G)\cong (\mathbb{Z}^2<\mathbb{Z}^2\times G)$ has relative Property (T), we know that $G\curvearrowright \widehat{k\mathbb{Z}^2}$ is rigid. To see that $G\curvearrowright \widehat{k\mathbb{Z}^2}/{\sim}$ is also rigid, we argue as follows. Set $\mathsf{B}=L^{\infty}(\widehat{k\mathbb{Z}^2}/{\sim})$, $\bN_0=\mathsf{B}\rtimes G$ and $\bN=L^{\infty}(\widehat{k\mathbb{Z}^2})\rtimes G$. Then, as $\mathsf{B}< L^{\infty}(\widehat{k\mathbb{Z}^2})< \bN$ and the inclusion $L^{\infty}(\widehat{k\mathbb{Z}^2})< \bN$ is rigid as the above shows, we deduce that the inclusion $\mathsf{B}< \bN$ is also rigid. 
	Moreover, by considering the action $G\times \mathbb{Z}/2\mathbb{Z}\curvearrowright \widehat{k\mathbb{Z}^2}$ as above, we see that $\bN_0=(L^{\infty}(\widehat{k\mathbb{Z}^2}))^{\mathbb{Z}/2\mathbb{Z}}\rtimes G=\bN^{\mathbb{Z}/2\mathbb{Z}}$, hence $[\bN: \bN_0]=[\bN: \bN^{\mathbb{Z}/2\mathbb{Z}}]<\infty$. Moreover the inclusion $\bN_0\subset\bN$ is $\frac{1}{2}$-Markov in the sense described in \cite{popa_book}. This can be seen by first identifying $\mathbb{T}^2$ with the set $[-\frac{1}{2}, \frac{1}{2}] \times [-\frac{1}{2}, \frac{1}{2}] $ and then choosing an explicit orthonormal basis $\{1,f\}$ for our inclusion, where $f$ is the function on $\mathbb{T}^2$ defined by the formula
\[ f(x,y)= \begin{cases} 1 & \textup{ for }  x+y \geq 0 \\ -1 &	\textup{ for }   x+y \leq 0 \end{cases}, \;\;\; (x, y) \in \mathbb{T}^2.\]
	Thus we can apply \cite[Proposition 4.6.2]{popa} to conclude that the inclusion $\mathsf{B}< \mathsf{N}_0$ is rigid, i.e.\ the action $G\curvearrowright \widehat{k\mathbb{Z}^2}/{\sim}$ is rigid. 
\end{proof}

\begin{cor}
	Let $G = SL_2(\bz)$ and let $\Lambda$ be an ICC group with the Haagerup property. Then $L(\Lambda \wr G)< L(\Lambda^{\oplus G} \times \bz^2)\rtimes G$ is a maximal Haagerup subalgebra.
\end{cor}

\begin{proof}
	An immediate consequence of Theorem \ref{thm: general version for 2nd example} and Lemma \ref{lem: (corrected version) quotients of GL_2(Z)-action}.
\end{proof}

\subsection*{Free products and pro-finite  actions}

The next result is connected to a recent work of Amrutam (\cite{Amr}, with the appendix by Amrutam and the first-named author) on intermediate algebras in the $\C^*$-context. The proof is motivated by the proof of \cite[Theorem 1.7]{packer}. Note that the results of Packer were later generalized (in particular dropping the assumptions of ergodicity and the existence of a normal conditional expectation) in \cite{suzuki}, using a different method.
\begin{thm} \label{thm:Amrutamgeneralised}
	Let $G\curvearrowright (X, \mu)$ be a p.m.p.\ ergodic action and $H< G$ be a strong relative ICC group, i.e.\ $\#\{hgh^{-1}: h\in H\}=\infty$ for all $g\neq e$ in $G$. Assume that $H$ acts on $X$ trivially. Then every intermediate von Neumann subalgebra $\bP$ between $L(G)$ and $L^{\infty}(X)\rtimes G$ is of the form $\bP=L^{\infty}(Y)\rtimes G$ for some $G$-factor $Y$ of $X$.
\end{thm}
\begin{proof}
	The strong relative ICC assumption implies that $G$ is ICC itself, so that by ergodicity of the action $L^{\infty}(X)\rtimes G$ is a II$_1$ factor. In particular we have the trace-preserving normal  conditional expectation $E: L^{\infty}(X)\rtimes G\twoheadrightarrow \bP$.
	
As in \cite{packer}, we just need to check that $E(L^{\infty}(X))\subseteq L^{\infty}(X)$. Indeed, if this holds, then  $E(L^{\infty}(X))$ is a $G$-invariant von Neumann subalgebra and $E(L^{\infty}(X))\rtimes G< \bP$; Moreover, $\bP< E(L^{\infty}(X))\rtimes G$ by applying $E$ to the Fourier expansion of an element in $\bP$. Therefore, $\bP=E(L^{\infty}(X))\rtimes G$ and $E(L^{\infty}(X))=L^{\infty}(Y)$ for some $G$-factor $Y$.

	To prove $E(L^{\infty}(X))\subseteq L^{\infty}(X)$, first observe that 	for every $a\in L^{\infty}(X)$, and $h\in H$, we have $u_hE(a)u_h^{-1}=E(h.a)=E(a)$ as $H$ acts on $X$ trivially. Thus
	$E(L^{\infty}(X))\subseteq L(H)'\cap \left(L^{\infty}(X)\rtimes G\right)$. Thus it suffices to show that the strong relative ICC assumption implies that $L(H)'\cap (L^{\infty}(X)\rtimes G)=L^{\infty}(X)$.
	Consider  $a\in L(H)'\cap \left(L^{\infty}(X)\rtimes G\right)$ and its Fourier expansion $a=\sum_{g\in G} f_g u_g$. Then for all $h\in H$ we have $u_ha=au_h$, which means that $f_{hgh^{-1}}=\sigma_h(f_g)=f_g$ for all $g\in G$.  Finally as $\sum_{g\in G}||f_g||_2^2<\infty$, the relative ICC assumption allows us to conclude that $f_g=0$ for all $g\neq e$.
	This implies that $E(L^{\infty}(X))\subseteq L^{\infty}(X)$ and ends the proof. 
\end{proof}

\begin{remark}
	The above theorem generalizes the one in \cite{Amr}. Indeed, one just observes that if $H<G$ is plump in the sense of \cite{Amr}, then it satisfies the strong relative ICC assumption. Indeed, suppose this is not the case, so that there exists some $ g\in G\setminus\{e\}$ such that $\#\{hgh^{-1}: h\in H\}<\infty$. Enumerate $\{hgh^{-1}: h \in H\}$ as $\{g_1,\ldots, g_N\}$ and define a function $\phi:X_N\to \mathbb{R}^{\geq 0}$, where $X_N=\{(t_1, \ldots,t_N):t_1,\ldots, t_N \geq 0, \sum_{i=1}^N t_i = 1\}$ by $\phi(t_1,\ldots, t_N)=||\sum_{i=1}^Nt_iu_{g_i}||$. Clearly, $\phi$ is continuous. Moreover the fact that $H$ is plump in $G$, means that $0=\inf_{\mathbf{t} \in X_N} \phi(\mathbf{t})$. Then, as $\phi$ is continuous, there is some $\mathbf{t}=(t_1,\ldots, t_N)\in X_N$ such that $\phi(\mathbf{t})=0$, i.e.\ $\sum_{i=1}^nt_iu_{g_i}=0$. This implies that $t_i=0$ for all $i=1,\ldots,N$ as $u_{g_1}, \ldots, u_{g_N}$ are linearly independent,  and contradicts the fact that $\mathbf{t}\in X_N$.
\end{remark}

\begin{cor}\label{prop: free product verison of main thm} 
Consider the action $SL_2(\mathbb{Z})*SL_2(\mathbb{Z})\curvearrowright \mathbb{Z}^2$ defined by first factoring onto the first copy of the free product: $SL_2(\mathbb{Z})*SL_2(\mathbb{Z})\overset{p_1}{\twoheadrightarrow} SL_2(\mathbb{Z})\curvearrowright \mathbb{Z}^2$. Then 	$L(SL_2(\mathbb{Z})*SL_2(\mathbb{Z}))<L(\mathbb{Z}^2\rtimes (SL_2(\mathbb{Z})*SL_2(\mathbb{Z})))$ is a maximal Haagerup subalgebra.
\end{cor}
\begin{proof}
Note first that the action of $SL_2(\mathbb{Z})$ (so also of $SL_2(\mathbb{Z})* SL_2(\mathbb{Z})$)	on $\mathbb{T}^2$ is p.m.p.\ and ergodic. Let $H=ker(p_1)$, i.e. the normal subgroup generated by the second copy of $SL_2(\bz)$, so $H$ acts trivially on $\bz^2$. Clearly $H \subset SL_2(\bz)*SL_2(\bz)$ is a strong relative ICC group. Hence by Theorem \ref{thm:Amrutamgeneralised} any von Neumann algebra $\bP$ such that $L(SL_2(\mathbb{Z})*SL_2(\mathbb{Z}))<\bP < L(\mathbb{Z}^2)\rtimes (SL_2(\mathbb{Z})*SL_2(\mathbb{Z})))$ is of the form $E(\bP)\rtimes (SL_2(\mathbb{Z})*SL_2(\mathbb{Z}))$, where $E$ is the trace preserving conditional expectation onto $L(\mathbb{Z}^2)$. If the subalgebra $E(\bP)$ is nontrivial, then, as it is $SL_2(\bz)$ invariant, by Lemma \ref{lem: (corrected version) quotients of GL_2(Z)-action} the algebra 
$ E(\bP)\rtimes SL_2(\mathbb{Z})$ is not Haagerup, and the inclusion $\bP \supset  E(\bP)\rtimes SL_2(\mathbb{Z})$ ends the proof. 
\end{proof}

We continue this part by explaining how Theorem \ref{thm:Amrutamgeneralised} and its proof can be used to determine intermediate von Neumann algebras for pro-finite actions of ICC groups.
Begin by noting that the proof of Theorem \ref{thm:Amrutamgeneralised} can be adopted to `localise' the necessary assumptions.

\begin{thm} \label{thm:Amrutamlocalised}
	Let $\alpha: G\curvearrowright (X, \mu)$ be a p.m.p.\ ergodic action. Let $A\subseteq L^{\infty}(X)$ be a subset such that $L^{\infty}(X)\subseteq\overline{span ~A}^{||\cdot||_2}$. Assume further that for each $a\in A$ there exists a subgroup $H_a\subseteq G$ such that
	\begin{rlist}
	\item $\alpha_{h_a}(a)=a$ for all $h_a\in H_a$;
	\item $\#\{hgh^{-1}: h\in H_a\}=\infty$ for all $g\neq e$ in $G$.
	\end{rlist}
	Then every intermediate von Neumann subalgebra $\bP$ between $L(G)$ and $L^{\infty}(X)\rtimes G$ is of the form $\bP=L^{\infty}(Y)\rtimes G$, where $Y$ is a $G$-factor of $X$.
\end{thm}

\begin{proof}
	The proof follows similarly to the proof of Theorem \ref{thm:Amrutamgeneralised}. 
	Indeed, observe that as the conditional expectation $E$ is normal, to show $E(L^{\infty}(X))\subseteq L^{\infty}(X)$ it suffices  to prove that $E(A)\subseteq L^{\infty}(X)$. Take then any $a\in A$. By assumption (i) we  know that $E(a)\in L(H_a)'\cap (L^{\infty}(X)\rtimes G)$. Then assumption (ii) on $H_a$ implies, as before, that $L(H_a)'\cap (L^{\infty}(X)\rtimes G)\subseteq L^{\infty}(X)$. 
\end{proof}

\begin{cor}
Let $G$ be a residually finite ICC group $G$, and let $\{G_n\}_{n \in \bn}$ be a decreasing sequence of normal subgroups of $G$ with trivial intersection.	Suppose that \[G\curvearrowright X:=lim_{\leftarrow}(G/G_n, \mu_n)\] is a profinite action. Then every von Neumann algebra $\bP$ such that $L(G)< \bP < L^{\infty}(X)\rtimes G$ is of the form $\bP=L^{\infty}(Y)\rtimes G$ where $Y$ is a $G$-factor of $X$.
\end{cor}

\begin{proof}
	Take $A=\{1_{gG_n}: n\geq 1, g\in G\}$. The fact that the span of $A$ is $\|\cdot\|_2$ dense in $L^\infty(X)$ follows from the definition of a profinite action.  For any $n\geq 1, g\in G$ and $a=1_{gG_n}$, let $H_a=G_n$.
	Then for any $n \in \bn$ and  $g\neq e$ in $G$ we have, denoting by $C_H(g)$ the centralizer of $g$ inside $H$, $\#\{hgh^{-1}: h\in G_n\}=[G_n: C_{G_n}(g)]=[G: C_{G_n}(g)]/[G: G_n]\geq [G: C_G(g)]/[G: G_n]=\infty$, as $G$ is assumed to be ICC. This means that we can apply Theorem \ref{thm:Amrutamlocalised}.
\end{proof}

The above result can be generalised from profinite actions to arbitrary compact actions, as we were kindly informed by R\'emi Boutonnet. After this work was completed, we learned that Chifan and Das proved a more general version of 
the above corollary, see \cite[Theorem 3.10]{chifan_das}, and used it, together with other results, to characterise  intermediate finite index subfactors for the inclusion $L(\Gamma) < L^{\infty}(X) \rtimes \Gamma$ for a free ergodic p.m.p.\ action,  and to provide an alternative proof of a version of Ioana's orbit equivalence superrigidity result \cite[Theorem A]{i_duke}.

\subsection*{Examples coming from roughly normal subgroups}

The next class of examples is related to what we call roughly normal subgroups and uses a variation on the work on intermediate operator algebras due to Cameron and Smith (\cite{cs_2}). It is worth mentioning that in the von Neumann algebraic context ideas similar to these below occur already in \cite{Haga}. 

If $H<G$ is an infinite subgroup then we call $H$ \emph{roughly normal} if for every $g \in G$ the set $H\cap g^{-1}Hg$ is infinite. A word of warning is in place: some authors call such subgroups almost normal, but it seems that in  the group theoretic terminology the latter usually means a subgroup with finitely many conjugate subgroups -- and the two notions are not related.

\begin{lem}\label{lem: intermediate subalgebras split}
	Let $H\subseteq G$ be a roughly normal subgroup and let $G\overset{\sigma}{\curvearrowright} (X,\mu)$ be a free  mixing p.m.p.\ action. Suppose that $\bP$ is a von Neumann algebra such that $L^{\infty}(X)\rtimes H< \bP<L^{\infty}(X)\rtimes G$. Then $P=L^{\infty}(X)\rtimes K$ for some group $K$ with $H< K<G$.
\end{lem}

\begin{proof}
	This is essentially a corollary of \cite[Theorem 3.3]{cs_2}. For completeness, we sketch the proof  in our setting.
	
Let $E: L^{\infty}(X)\rtimes G\twoheadrightarrow \bP$ be a faithful normal conditional expectation. Notice that for every $g \in G$ we have $u_g^*E(u_g)\in L^{\infty}(X)'\cap (L^{\infty}(X)\rtimes G)=L^{\infty}(X)$ as the action is free. We may thus write $E(u_g)=z_gu_g$ for some $z_g\in L^{\infty}(X)$. Clearly, $z_h=1$ for all $h\in H$. Therefore (see \cite{cs_2}) $\bP=\overline{span}^{w^*}\{L^{\infty}(X)z_gu_g: g\in G\}$.
	
	Now we show that $z_g$ are projections  and moreover $z_g\sigma_g(z_h)=z_gz_{gh}$, $z_{g^{-1}}=\sigma_{g^{-1}}(z_g^*)$ and $z_e=1$ for all $g, h\in G$. Indeed, since $z_{g^{-1}}u_{g^{-1}}=E(u_{g^{-1}})=E(u_g)^*=(z_gu_g)^*=\sigma_{g^{-1}}(z_g^*)u_{g^{-1}}$, we get $\sigma_{g^{-1}}(z_g^*)=z_{g^{-1}}$. Further note that $z_gz_{gh}u_{gh}=E(z_gu_{gh})=E(E(u_g)u_h)=E(u_g)E(u_h)=z_gu_gz_hu_h=z_g\sigma_g(z_h)u_{gh}$, hence $z_gz_{gh}=z_g\sigma_g(z_h)$. Taking $h=e$, we deduce that $z_g^2=z_g\in L^{\infty}(X)$, hence $z_g$ is an orthogonal projection.
	
Define $K=\{g\in G: z_g=1\}$. Clearly, $K$ is a subgroup of $G$ containing $H$. 
Take then any $g\in G\setminus K$. As $H$ is roughly normal in $G$, there exist infinitely many distinct $h_n\in H$ such that $h_n':=gh_ng^{-1}\in H$. Since $z_{h_n}=1$, we get $\sigma_{h_n}(z_{g^{-1}})=z_{h_ng^{-1}}=z_{g^{-1}h_n'}\geq z_{g^{-1}}$, where the last inequality holds since $z_{g^{-1}}z_{g^{-1}h_n'}=z_{g^{-1}}\sigma_{g^{-1}}(z_{h_n'})=z_{g^{-1}}.$ As $G\curvearrowright X$ is mixing and $z_{g^{-1}}\neq 1$, we deduce that $z_{g^{-1}}=0$. 
This means that  $\bP=\overline{span}^{w^*}\{L^{\infty}(X)u_g: g\in K\}$, which ends the proof.
\end{proof}

An extension of the above argument yields the next theorem.

\begin{thm}
	Let $G=\mathbb{Z}^2\rtimes SL_2(\mathbb{Z})$ and $H=\mathbb{Z}^2\rtimes C$, where $C< SL_2(\mathbb{Z})$ is a maximal amenable subgroup. Let $\sigma: G\curvearrowright (X, \mu)$ be a free p.m.p.\ action such that $\sigma|_{\mathbb{Z}^2}$ is ergodic. Then $L^{\infty}(X)\rtimes H<L^{\infty}(X)\rtimes G$ is a maximal Haagerup subalgebra.
\end{thm}

\begin{proof}
	Notice first that $H$ is a roughly normal subgroup of $G$. Since the action is not assumed to be mixing, we cannot apply Lemma \ref{lem: intermediate subalgebras split} directly. If we however follow its proof, and use the properties of our groups, we see that  for any $g\in G\setminus K\subseteq G\setminus H$ (with $K$ defined as in that proof) and any $h\in \bz^2$ we have $h'=ghg^{-1}\in H$ and deduce that $\sigma_{h}(z_{g^{-1}})\geq z_{g^{-1}}$. This means that actually $\sigma_h(z_{g^{-1}})=z_{g^{-1}}$ for all $h\in \mathbb{Z}^2$. Since $\sigma|_{\mathbb{Z}^2}$ is ergodic and $z_{g^{-1}}\neq 1$, we deduce again that $z_{g^{-1}}=0$. Hence every intermediate von Neumann subalgebra must be of the form $L^{\infty}(X)\rtimes K$ for a subgroup $K$ of $\bz^2 \rtimes SL_2(\mathbb{Z})$ containing $H$. Now Theorem \ref{prop: classify mHAP in z2 times sl2z} ends the proof.
\end{proof}
We can now present concrete examples of the above situation.

\begin{cor}
	Let $G=\mathbb{Z}^2\rtimes SL_2(\mathbb{Z})$ and $H=\mathbb{Z}^2\rtimes C$, where $C< SL_2(\mathbb{Z})$ is a maximal amenable subgroup. Let $\sigma$ be the classical Bernoulli shift $G\curvearrowright \mathbb{T}^G$ or a generalized Bernoulli shift $G\curvearrowright \mathbb{T}^{\mathbb{Z}^2}$ induced by the affine action $G\curvearrowright \mathbb{Z}^2$. Then $\sigma$ is free and $\sigma|_{\mathbb{Z}^2}$ is ergodic; hence $L^{\infty}(X)\rtimes H<L^{\infty}(X)\rtimes G$ is a maximal Haagerup subalgebra. In particular for example $L(\bz \wr_G H)<L(\bz \wr G)$ is a maximal Haagerup subalgebra. 
\end{cor}

Note that the generalized Bernoulli shift action is not mixing as $(1, 0)\in \mathbb{Z}^2$ has infinite stabilizer subgroup in $G$.

\subsection*{Questions of Ge}
The following natural definition of a maximal von Neumann subalgebra was introduced by Ge \cite{ge}.

\begin{definition}
Let $\bM$ be a von Neumann algebra and $\bN$ be a von Neumann subalgebra. We say $\bN$ is a \emph{maximal} von Neumann subalgebra  of $\bM$ if for any von Neumann subalgebra $\bP$ of $\bM$ with $\bN\subseteq \bP$, either $\bN=\bP$ or $\bP=\bM$. 
\end{definition}

In \cite[Section 3, Question 2]{ge} Ge asked the following questions.
\begin{quest} \label{q:Ge}
	Can a non-hyperfinite factor of type II$_1$ have a hyperfinite factor as a maximal von Neumann subalgebra? Can a maximal von Neumann subalgebra of the hyperfinite factor of type II$_1$ be a subfactor of an infinite Jones index? \end{quest}

We will now show how the knowledge of intermediate von Neumann algebras quoted and developed in this section gives positive answers to both of the above. We begin with a lemma on properties of upper-triangular matrices inside $SL_2(\mathbb{Q})$. Recall that the notion of a roughly normal subgroup was introduced before Lemma \ref{lem: intermediate subalgebras split}.

\begin{lem}\label{lem:SLQ}
	Let $G=SL_2(\mathbb{Q})$ and let $H$ denote the upper-triangular matrices in $SL_2(\mathbb{Q})$. Then $H$ is a maximal subgroup in $G$; moreover $H$ is amenable and roughly normal in $G$.
\end{lem}

\begin{proof}
	Put $s = \begin{pmatrix} 0 & -1 \\ 1 & 0\end{pmatrix}$. Suppose that $g \in G\setminus H$ and write $g =\begin{pmatrix} a & b \\ c & d\end{pmatrix}$, where $a,b,c,d \in \bq$, $c \neq 0$. Then 
	\[g = \begin{pmatrix} 1 & a/c \\ 0 & 1\end{pmatrix}   \begin{pmatrix} 0 & -1 \\ 1 & 0\end{pmatrix} \begin{pmatrix} c & d \\ 0 & 1/c\end{pmatrix}.\]
This means that $G = H \sqcup Hs H$, so that $H$ is a maximal subgroup.	As $H$ is solvable, it is amenable. Finally the set 
\[\left\{\begin{pmatrix} x & 0 \\ 0 & 1/x\end{pmatrix}:x\in \bq, x \neq 0   \right\}\]
is infinite and contained in $sHs^{-1} \cap H$, so that $H$ is roughly normal in $G$.
\end{proof}

\begin{prop}
		Let $G=SL_2(\mathbb{Q})$, let $H$ denote the upper-triangular matrices in $SL_2(\mathbb{Q})$ and let $\Lambda$ be an amenable ICC group. The following factor inclusions have the property that the smaller factor is an amenable maximal subalgebra of the larger one:
	\begin{enumerate}
	\item $L(\Lambda \wr_G H ) < L(\Lambda \wr G)$;
	\item $L(\bz \wr_G H ) < L(\bz \wr G)$.
	\end{enumerate}
\end{prop}

\begin{proof}
All the groups in sight are ICC; the smaller ones are amenable by Lemma \ref{lem:SLQ}, the larger ones are non-amenable as $G$ is not amenable. To see that all the intermediate von Neumann algebras must come from intermediate subgroups between $H$ and $G$ in the first case we can apply \cite[Corollary 4]{choda} exactly as was done in the proof of Theorem \ref{thm:Galois}, and in the second case appeal to Lemma \ref{lem: intermediate subalgebras split}. Another application of Lemma \ref{lem:SLQ} ends the proof. 	
\end{proof}

For the second part of Question \ref{q:Ge} the positive answer was given already by Suzuki in  \cite[Example 4.14]{suzuki}; we recall again Suzuki's example and present another one, using only group von Neumann algebras.  

Recall that  a p.m.p.\ action $\mathbb{Z}\curvearrowright X$ is called prime if it has no nontrivial proper factors. For existence of such actions (e.g.\ a Cha\`con system), see \cite[Theorem 16.6]{glasner}. The result below follows as in the proof of  Theorem \ref{thm:Galois}.

\begin{prop}[\cite{suzuki}]
Let $\Lambda$ be an ICC amenable group. If  $\mathbb{Z}\curvearrowright X$ is a prime action, then the following infinite index inclusion of amenable factors is such that the smaller factor is a maximal von Neumann subalgebra of the larger one:
\[ L(\Lambda\wr \bz) <(L(\Lambda^{\oplus \bz}) \otimes L^\infty(X))\rtimes \bz. \]	
\end{prop}

To present the second example, we need some preparations. The so-called Houghton groups were introduced in \cite{houghton}. Let us recall their definition,  following \cite[Example 3.6]{de2}, .

Fix an integer $n\in \bn$ and set $\Omega_n=\mathbb{N}\times \{1,\ldots, n\}$. We may think of $\Omega_n$ as the disjoint union of $n$ copies $\mathbb{N}_1,\ldots, \mathbb{N}_n$ of $\mathbb{N}$. The Houghton group $G_n$ is the group of all permutations $\sigma$ of $\Omega_n$ such that, for each $i\in \{1, \ldots,n\}$ the set $\sigma(\mathbb{N}_i)\Delta \mathbb{N}_i$ is finite, and $\sigma$ is eventually a translation on $\mathbb{N}_i$, i.e.\ there exist an $n$-tuple $(m_1,\ldots, m_n)\in\mathbb{Z}^n$ and a finite set $K_{\sigma}\subseteq \Omega_n$ such that $\sigma(k,i)=(k+m_i, i)$ for all $(k,i)\in \Omega_n\setminus K_{\sigma}$. It is easy to see that the action of $G_n$ on $\Omega_n$ is transitive.
Note that $G_1=S_\infty$.

\begin{prop}
 Let $n \in \bn$, and let $G_n$ denote the corresponding Houghton group acting on $\Omega_n$ as above, let $H_n$ denote the stabilizer group of a point in $\Omega_n$ 	and let $\Lambda$ be an ICC amenable group. The following infinite index inclusion of amenable factors is such that the smaller factor is a maximal von Neumann subalgebra of the larger one:
	\[ L(\Lambda\wr_{G_n} H_n) <L(\Lambda \wr G_n). \]	
\end{prop}
\begin{proof}
	Let $n \in \mathbb{N}$. As explained in \cite[Example 3.6]{de2}, $G_n$ is elementary amenable. Moreover, the action of $G_n$ on $\Omega_n$ is $k$-transitive for all $k\in \bn$.  Taking $k=2$ we see that the diagonal action $G_n\curvearrowright \Omega_n\times \Omega_n$ has two orbits; equivalently, $|H_n/ G_n\backslash H_n|=2$. Hence $H_n$ is a maximal subgroup of $G_n$. It has infinite index, as $G_n\curvearrowright \Omega_n$ is transitive and $\Omega_n$ is an infinite set.	
The conclusion follows once again by \cite[Corollary 4]{choda}, as in the proof of Theorem \ref{thm:Galois}.	
\end{proof}

\section{Maximal (T) and non-(T) subgroups and subalgebras}

In this short section we discuss some facts concerning maximal non-(T) (and also (T)) subgroups and subalgebras.

\subsection*{Explicit maximal non-(T) subgroups in groups with Property (T)}

Variations of the example to be described below were studied for example in \cite{iv,dv}.

\begin{proposition} \label{prop:maxnonT}
	Consider the group 
$$G=\Bigg\{\begin{pmatrix}
1&b&c\\
0&A&d\\
0&0&1\\
\end{pmatrix}~\bigm\vert\ A\in SL_3(\mathbb{Z}), b\in \mathbb{Z}^{1\times 3}, c\in \mathbb{Z}, d\in\mathbb{Z}^{3\times 1}\Bigg\}.
$$
and its subgroup
$$H=\Bigg\{\begin{pmatrix}
1&0&c\\
0&A&d\\
0&0&1\\
\end{pmatrix}~\bigm\vert\ A\in SL_3(\mathbb{Z}), c\in \mathbb{Z}, d\in\mathbb{Z}^{3\times 1}\Bigg\}.$$
Then $H$ is a maximal non-(T) subgroup inside $G$.
\end{proposition}
\begin{proof}
	Note first that the fact that $G$ is a Kazhdan group is observed in \cite{de}. If we consider $$K:=\Bigg\{\begin{pmatrix}
	1&0&0\\
	0&A&d\\
	0&0&1\\
	\end{pmatrix}~\bigm\vert\ A\in SL_3(\mathbb{Z}), d\in\mathbb{Z}^{3\times 1}\Bigg\},$$
then a direct computation shows that   $K$ is a normal subgroup of $H$ and $H/K\cong \mathbb{Z}$. Hence  $H$ is non-(T), as it admits a non-(T) quotient.

We will now show that for any $g\in G\setminus H$ the subgroup $\langle H, g \rangle$ has finite index in $G$ (and hence has Property (T); and so does any subgroup of $G$ containing $\langle H,g\rangle$).  

Consider the normal subgroup $N$ of $G$ given by 
$$N=\Bigg\{\begin{pmatrix}
1&0&c\\
0&I&d\\
0&0&1\\
\end{pmatrix}~\bigm\vert\ c\in \mathbb{Z}, d\in\mathbb{Z}^{3\times 1}\Bigg\}.$$
We have naturally $N\subset H$, so that it suffices to show that $\langle H, g \rangle/N$ has finite index in $G/N$. This however allows us to reduce the dimension. Put 
\[ \tilde{H} = \Bigg\{\begin{pmatrix}
1&0\\
0&A
\end{pmatrix}~\bigm\vert\ A\in SL_3(\mathbb{Z}) \Bigg\},\]
\[ \tilde{g} = \begin{pmatrix}
1&b\\
0&I
\end{pmatrix},\] where $0\neq b=n(b_1, b_2, b_3)$ with $b_1,b_2,b_3, n \in \bz$, $n\neq 0$, $gcd(b_1, b_2, b_3)=1$. It suffices to show that the subgroup $\langle \tilde{H}, \tilde{g} \rangle$ is of finite index in $\tilde{G} = \begin{pmatrix}
1&\bz^{1 \times 3}\\
0&SL_3(\mathbb{Z}) \end{pmatrix}.$

But $\tilde{G}\cong \mathbb{Z}^3\rtimes SL_3(\mathbb{Z})$ and $\tilde{H}$ is identified as the copy of $SL_3(\mathbb{Z})$ in $\mathbb{Z}^3\rtimes SL_3(\mathbb{Z})$, thus, for any $\tilde{g}\in \tilde{G}\setminus \tilde{H}$, we must have $\langle \tilde{H}, \tilde{g}\rangle \cong n\mathbb{Z}^3\rtimes SL_3(\mathbb{Z})$ for some $n\neq 0$, so that $\langle \tilde{H}, \tilde{g}\rangle$  has finite index in $\mathbb{Z}^3\rtimes SL_3(\mathbb{Z})$.

\end{proof}

In fact another example of similar nature may be deduced from the results in \cite{meiri} and \cite{venk}, as kindly communicated to us by Chen Meiri.

\begin{proposition}\label{prop: max non-(T) in sl3z}
Let $n\geq 3$ and define \[ \widetilde{H}=\left\{\begin{pmatrix}
	A&B\\
	0& C
	\end{pmatrix}:~ A\in GL_2(\mathbb{Z}),  B\in M_{2\times(n-2)}(\bz), C \in GL_{n-2}(\bz) \right\}~\mbox{and}~H=\widetilde{H}\cap SL_n(\bz).\] Then $H$ (resp. $\widetilde{H}$) is a maximal non-(T) subgroup in $SL_n(\bz)$ (resp. $GL_n(\bz)$).
\end{proposition}

\begin{proof}
Assume we have proved $H$ is a maximal non-(T) subgroup in $SL_n(\bz)$, then $\widetilde{H}$ is also a maximal non-(T) subgroup in $GL_n(\bz)$. Indeed, $\widetilde{H}$ is non-(T) as $[\widetilde{H}: H]=2$. Moreover, if $\widetilde{H}\lneq \widetilde{K}<GL_n(\bz)$, then one must have $H\lneq \widetilde{K}\cap SL_n(\bz)$ by index considerations. Hence $\widetilde{K}\cap SL_n(\bz)$ has (T), which implies $\widetilde{K}$ has property (T) as $[\widetilde{K}: \widetilde{K}\cap SL_n(\bz)]=2$. So we only deal with $H$ below.

	Note first that $H$ does not have Property (T) since it has a quotient isomorphic to $GL_2(\bz)$, which does not have Property (T).
	
We will first give a proof in a special case of $n=3$, where it is fully elementary.	
Let then $g\in SL_3(\bz), g\not\in H$. Multiplying $g$ by elements of $H$ and using basic number-theoretic properties one can first reduce the situation to the case where $g_{31}=0$ (so that $g_{32}\neq 0$) and then in addition to the case where also $g_{21}\neq 0$. Thus $g = \begin{pmatrix}  * & * & * \\ a_1&  * & * \\ 0 &a_2 & *    \end{pmatrix}$ with $a_1 a_2 \neq 0$.
Now the proof of \cite[Theorem 2]{meiri} shows that $[SL_3(\mathbb{Z}): \langle g, Id+e_{1, 3} \rangle]<\infty$, where $e_{1, 3}$ denotes the matrix unit with 1 at the $(1, 3)$-entry. Hence $[SL_3(\mathbb{Z}): \langle H, g \rangle]<\infty$ as $Id+e_{1, 3}\in H$, and $\langle H, g\rangle$ (as well as any subgroup containing it) must have Property (T).

Consider then the general case. 

The group  $H$ consists of the integral points of a maximal parabolic subgroup of $SL_n(\mathbb{Q})$ (see for example the discussion in \cite[p. 86--87]{mar}). Thus, if $g \in SL_n(\bz)\setminus H$, then $\langle H,g\rangle$ is Zariski-dense in $SL_n(\mathbb{Z})$. It is easy to see that $H$ admits two unipotent elements which generate a copy of $\bz^2$. It then follows from Venkataramana's  result \cite[Theorem 3.7]{venk} that  $\langle H,g\rangle$ is of finite index in $SL_3(\bz)$ and thus has property (T). 

 \end{proof}

\subsection*{Maximal (T) and non-(T)-subalgebras}

As we saw in Proposition \ref{maximalNonT}, it is very easy to show that maximal non-(T) groups exist. We do not know how to extend this result to general von Neumann algebras, but below we record one special case. For the notions related to the von Neumann algebraic Property (T)  and to the relative Property (T) we refer again to \cite{popa}.

\begin{proposition}
	Let $\bM$ be a II$_1$ factor and $\bN< \bM$ be a non-(T) von Neumann subalgebra, which is ireducible, i.e.\ $\bN'\cap \bM=\mathbb{C}$. If $\bM$ has property (T), then there is a maximal non-(T) von Neumann subalgebra $\bP$ such that $\bN< \bP< \bM$.
\end{proposition}
\begin{proof}
	Consider the class of non-(T) von Neumann subalgebras of $\bM$ which contain $\bN$, as usual partially ordered by inclusion. To conclude the proof via the the Kuratowski-Zorn lemma it suffices to show that for any ascending chain $(\bN_i)_{i \in \mathcal{I}}$ in the class, the von Neumann algebra $\bN_{\infty}:=(\cup_{i \in \mathcal{I}}\bN_i)''$ is in the class. As we are working inside a II$_1$-factor we may assume that the index set is countable.
	Note that $\bN_{\infty}$ is a factor and $\bN_i'\cap \bM=\mathbb{C}$ for each $i \in \mathcal{I}$, since $\bN$ is assumed to be irreducible.  
	Suppose that $\bN_{\infty}$ has Property (T). Then \cite[Theorem 4.4.1]{popa_correspondence} implies that $\bN_{\infty}=\bN_i$ for some $i \in \mathcal{I}$, which yields a contradiction. This ends the proof.
\end{proof}

As noted before Proposition \ref{maximalNonT} one cannot expect a general result of this type for Property (T). Having said that, using free products and one can exhibit explicit examples of maximal Property (T) subgroups/subalgebras.

\begin{proposition}
	Let $\bM, \bN$ be type II$_1$ factors. If $\bM$ has Property (T), then $\bM< \bM*\bN$ is a maximal rigid embedding, i.e.\ if $\bP$ is any von Neumann algebra with $\bM< \bP< \bM *\bN$ and $\bP< \bM*\bN$ is a rigid embedding, then $\bP=\bM$. In particular $\bM$ is a maximal (T) von Neumann subalgebra in $\bM*\bN$.
\end{proposition}

\begin{proof}
	It suffices to prove the first part since if a von Neumann subalgebra $\bP<\bM$ has Property (T), then $\bP< \bM*\bN$ is a rigid embedding.
	
	If $\bM<\bP$ and $\bP< \bM*\bN$ is a rigid embedding, then $\bP$ is diffuse since $\bM$ is diffuse. By \cite[Theorem 5.1]{ipp} (taking $\mathsf{B}=\mathbb{C}=\bM_0$ there), there exists a unique pair of projections $q_1, q_2\in \bP'\cap (\bM*\bN)$ such that $q_1+q_2=1$,  $u_1(\bP q_1)u_1^*\subseteq \bM$ and $u_2(\bP q_2)u_2^*\subseteq \bN$ for some unitary elements $u_1, u_2\in\bM*\bN$.
	Since $\bP'\cap (\bM*\bN)< \bM'\cap (\bM*\bN)=\bM'\cap \bM=\mathbb{C}$, either $(q_1, q_2)=(0, 1)$ or $(q_1, q_2)=(1, 0)$.
	If $(q_1, q_2)=(0,1)$, then $u_2 \bP u_2^*< \bN$ and hence $u_2\bM u_2^*< \bN$. Then by \cite[Theorem 1.1]{ipp} it follows that $u=0$, which is  a contradiction. 
	Thus $(q_1, q_2)=(1, 0)$ and $u_1\bM u_1^*< u_1\bP u_1^*< \bM$. Again, by \cite[Theorem 1.1]{ipp}  it follows that $u_1^*\in L^2(\bM)$, hence $u_1^*\in \bM$. Now, $\bM< \bP< u_1^*\bM u_1=\bM$, i.e.\ $\bM=\bP$.
\end{proof}
	
\begin{cor}
Suppose that $G,H$ are ICC groups and $G$ has Property (T). Then $G$ is a maximal Property (T) subgroup of $G*H$.	
\end{cor}
	\begin{proof}
Immediate from the last proposition.		
\end{proof}

We finish the section by exhibiting a concrete example of a maximal non-(T) von Neumann subalgebra inside the II$_1$ factor with Property (T), based on recent results of Kaluba, Nowak and Ozawa \cite{kno} and Kaluba, Kielak and Nowak \cite{kkn}, together with Proposition \ref{prop: max non-(T) in sl3z}. In the first version of our paper we asked for such explicit examples: and  Chifan, Das and Khan showed later in \cite{chifan_das_khan}, independently of our construction below, that they can be obtained with the help of the group-theoretic Rips construction.

Once again we begin by some group-theoretic observations.

\begin{lem}\label{lem: Inn<Aut is relative ICC} Let $n \in \bn$, $n \geq 2$.
Then 	$Inn(\mathbb{F}_n)$ is relative ICC in $Aut(\mathbb{F}_n)$, i.e.\ $\#\{s\phi s^{-1}: s\in Inn(\mathbb{F}_n)\}=\infty$ for all $id\neq \phi\in Aut(\mathbb{F}_n)$.
\end{lem}
\begin{proof}
	Take any nontrivial $\phi\in Aut(\mathbb{F}_n)$. Write $\textup{Fix}(\phi)=\{g\in \mathbb{F}_n: \phi(g)=g\}$. 
	
	We claim first that $[\mathbb{F}_n: \Fix(\phi)]=\infty$. Indeed, suppose this is not the case and  say $[\mathbb{F}_n: \Fix(\phi)]=k\in \bn$. Then for any nontrivial element $g\in \mathbb{F}_n$, $\phi(g)^k=g^k$. Fix such an element and note that the subgroup $\langle \phi(g), g\rangle$ is a free group, isomorphic either to $\mathbb{Z}$ or to $\mathbb{F}_2$ (since the minimal number of generators for the free group $\mathbb{F}_m$ is $m$). Clearly, $\langle \phi(g), g \rangle\not\cong \mathbb{F}_2$; otherwise, $\phi(g)$ and $g$ would have to be free, which contradicts the above relation. Therefore $\langle \phi(g), g \rangle\cong \mathbb{Z}=\langle s \rangle$, so that $\phi(g)=s^a, g=s^b$ for some nonzero $a, b\in \bz$. Then $s^{ak-bk}=e$, so $a=b$, i.e.\ $\phi(g)=g$. As $g$ was arbitrary,  $\phi$ must be trivial and we have reached a contradiction.

	We can thus find  an infinite sequence of elements $g_i\in \mathbb{F}_n$ such that $g_i\Fix(\phi)\cap g_j\Fix(\phi)=\emptyset$ for all $i, j \in \bn, i\neq j$.
To finish the proof	it suffices to check that $Ad(g_i)\circ \phi \circ Ad(g_i^{-1})\neq Ad(g_j)\circ \phi \circ Ad(g_j^{-1})$ for all $i,j\in \bn, i \neq j$.
	
	One can check that if $x \in G$ then  $Ad(g_i)\circ \phi\circ  Ad(g_i^{-1})(x)=(Ad(g_j)\circ \phi \circ Ad(g_j^{-1})(x)$ if and only if $\phi(g_j)g_j^{-1}g_i\phi(g_i^{-1})$ commutes with $\phi(x)$. Hence, as $\phi$ is an automorphism, if $Ad(g_i)\circ \phi \circ Ad(g_i^{-1})=Ad(g_j)\circ \phi \circ Ad(g_j^{-1})$, then 
	$$\phi(g_j)g_j^{-1}g_i\phi(g_i^{-1})=e, $$ i.e.\ $g_j^{-1}g_i\in \Fix(\phi)$, This contradicts the conditions imposed earlier on the elements $g_i$.
\end{proof}

Let $n \in \bn$.
 The group $Out(\mathbb{F}_n)$ acts  naturally on $\mathbb{F}_n/[\mathbb{F}_n, \mathbb{F}_n]\cong \mathbb{Z}^n$, which induces a group homomorphism $\phi: Out(\mathbb{F}_n)\to GL_n(\mathbb{Z}):=G$. 
Moreover, we have a short exact sequence $0\to Inn(\mathbb{F}_n)\cong \mathbb{F}_n\to Aut(\mathbb{F}_n)\overset{\pi}{\to} Out(\mathbb{F}_n)\to 0$.

\begin{lem}\label{lem:maxnonT}  Let $n \in \bn, n \geq 5$, and let $H_0:=\phi^{-1}(\widetilde{H})<Out(\mathbb{F}_n)$,
	where \[\widetilde{H}:=\left\{\begin{pmatrix}
	A &B\\
	0& C
	\end{pmatrix}: A\in GL_2(\mathbb{Z}), C\in GL_{n-2}(\mathbb{Z}), B\in M_{2, n-2}(\mathbb{Z})\right\}.\] Then $H_0$ is a maximal non-(T) subgroup in $Out(\mathbb{F}_n)$ with the extra property that $[Out(\mathbb{F}_n): \langle H_0, g \rangle]<\infty$ for all $g\in Out(\mathbb{F}_n)\setminus H_0$. Further $\pi^{-1}(H_0)$ is a maximal non-(T) subgroup of $Aut(\mathbb{F}_n)$.
\end{lem}

\begin{proof}
	It follows from Proposition \ref{prop: max non-(T) in sl3z} and its proof that $H:=\widetilde{H}\cap SL_n(\bz)$  is a maximal non-(T) subgroup of $SL_n(\bz)$ with the extra property that $[SL_n(\bz): \langle H, g \rangle]<\infty$ for all $g\in SL_n(\bz)\setminus H$. Clearly, this also implies $[GL_n(\bz): \langle \widetilde{H}, g\rangle]<\infty$ for all $g\in GL_n(\bz)\setminus \widetilde{H}$.  
	
	As Property (T) passes to quotients, we deduce that $H_0$ is non-(T).
	
	Let then $H_0\lneq K\lneq Out(\mathbb{F}_n)$ be any group. Then it is easy to check that $\widetilde{H}\lneq \phi(K)\lneq G$ as $ker(\phi)\leq H_0$.   We know that $[G: \phi(K)]<\infty$; therefore, $[Out(\mathbb{F}_n): K]\leq [\phi(Out(\mathbb{F}_n)): \phi(K)]\leq [G: \phi(K)]<\infty$ (as $ker(\phi)\leq K$ implies that $\phi^{-1}(\phi(K))=K$). Thus $K$ has Property (T) as $Out(\mathbb{F}_n)$ has Property (T) for all $n\geq 5$ by \cite{kkn, kno}.

The second statement follows in a very similar way.
\end{proof}

\begin{proposition}\label{prop:maxnonTvNa}
Let $n \in \bn, n \geq 5$ and let $\gamma: Aut(\mathbb{F}_n) \to GL_n(\bz)$ be the homomorphism obtained by the composition of $\pi: Aut(\mathbb{F}_n) \to Out(\mathbb{F}_n)$ and $\phi:Out(\mathbb{F}_n) \to GL_n(\bz)$ (see the discussion before Lemma \ref{lem:maxnonT}). Further let $\widetilde{H}<GL_n(\bz)$ be the subgroup defined in Lemma \ref{lem:maxnonT}. Then
$L(\gamma^{-1}(\widetilde{H}))$ is  a maximal non-(T) von Neumann subalgebra inside  $L(Aut(\mathbb{F}_n))$ (where the latter is a II$_1$ factor with Property (T)). 
\end{proposition}
\begin{proof}
	By  Lemma \ref{lem: Inn<Aut is relative ICC}  we know that $Inn(\mathbb{F}_n)$ is relative ICC in $\Gamma=Aut(\mathbb{F}_n)$, so that $L(Inn(\mathbb{F}_n))'\cap L(Aut(\mathbb{F}_n))=\mathbb{C}1$ and every intermediate von Neumann subalgebra between $L(Inn(\mathbb{F}_n))$ and $L(Aut(\mathbb{F}_n))$ is a subfactor. As $Inn(\mathbb{F}_n)$ is also normal in $Aut(\mathbb{F}_n)$, \cite[Corollary 3.8(2)]{chifan_das} implies that every intermediate subfactor between $L(Inn(\mathbb{F}_n))$ and $L(Aut(\mathbb{F}_n))$ is of the form $L(K)$ for some intermediate subgroup $Inn(\mathbb{F}_n)\leq K\leq Aut(\mathbb{F}_n)$. Lemma \ref{lem:maxnonT} ends the proof. 
\end{proof}


\section{Open problems}

We finish the article by listing certain open problems regarding the maximal Haagerup subgroups and subalgebras, accompanied by brief comments on what we know about them so far.

\begin{prob}
Find an example of a maximal Haagerup subgroup $H<G$ such that $L(H)$ is not a maximal Haagerup subalgebra inside $L(G)$.	
\end{prob}

For amenability the relevant examples were produced for example  in \cite{remi_carderi2}. We believe that a suitable candidate is given by the pair of groups considered in Proposition \ref{prop:wreathabelian}, although it is not clear whether one can find a Bernoulli factor of the corresponding algebraic action.

In the first version of this paper we asked whether one can find a  non-Haagerup group $G$ such that $\bz^n$ $(n\geq 1)$ is a maximal Haagerup subgroup of $G$ or show that no such example exists. Note that our results include examples of amenable and maximal Haagerup subgroups inside non-Haagerup groups. On the other hand,  observe that if $\bz$ can be realised as a maximal Haagerup subgroup inside a non-Haagerup group $G$, then $G$ does not have property $P_{\mbox{na\"i}}$ as introduced in \cite{bekch} and it admits an infinite cyclic or trivial Haagerup (hence also amenable) radical. As was pointed to us by Dan Ursu, the examples where $\bz$ is a maximal Haagerup group in a Property (T) group are provided by `torsion-free Tarski monsters', as constructed in \cite{Olshanski}. Indeed, Corollary 1 of that paper shows that any non-cyclic torsion free hyperbolic group (of which there are Property (T) examples) admits a non-abelian torsion free quotient whose all non-trivial subgroups are cyclic.

\begin{prob} \label{prob:semidirect}
Prove that $L(SL_2(\bz))$ is a maximal Haagerup subalgebra of $L(\bz^2 \rtimes SL_2(\bz))$ and more generally $L(\widehat{A}G\rtimes G)$ is a maximal Haagerup subalgebra inside  $L((\widehat{A}G\oplus \mathbb{Z}^2)\rtimes G)$, where $G=SL_2(\mathbb{Z})$ and $A$ is any finite abelian group.	
\end{prob}

To attack the second problem mentioned above, following the proof of Theorem \ref{thm: general version for 2nd example}, one may need to describe all intermediate factors $G\curvearrowright X$ such that there exist $G$-equivariant p.m.p.\ maps $\alpha$, $\beta:\mathbb{T}^2\times A^G\overset{\alpha}{\to}X\overset{\beta}{\to}A^G$ with $\beta\circ \alpha=proj_{A^G}$.

\begin{prob}
Determine all maximal Haagerup subgroups inside $SL_3(\bz)$. 
\end{prob}

This is naturally related to Proposition \ref{prop:maxHAPSL3Z}, Proposition \ref{prop: max non-(T) in sl3z} and Proposition \ref{prop:maxnonTvNa}.

\begin{prob}
Find an explicit example of a maximal non-(T) subalgebra in 
$L(SL_{n\geq 3}(\bz))$.
\end{prob}

Here of course one could  also ask specifically about the subgroups from Propositions \ref{prop:maxnonT} and \ref{prop: max non-(T) in sl3z}.

\begin{prob}
	Find explicit and natural examples of maximal Haagerup von Neumann subalgebras of type III.
\end{prob}

Here the situation seems to be completely open, in a sense that no natural obstructions to the Haagerup property seem to be known beyond the context of finite von Neumann algebras.

\subsection*{Acknowledgements} We acknowledge several useful discussions on the subject of this paper with R\'emi Boutonnet, Chen Meiri, Yuhei Suzuki,  Dan Ursu, Chenxu Wen, and especially with Stefaan Vaes. We also thank Ionut Chifan, Sayan Das, Ami Viselter and Mateusz Wasilewski for their comments. 
The authors were partially supported by the National Science Center (NCN) grant no.~2014/14/E/ST1/00525.

\subsection*{Note added in proof:} The first part of Problem \ref{prob:semidirect} has now been solved in \cite{Jiang}. 

\end{document}